\documentclass[12pt]{amsart}

\usepackage[dvipsnames]{xcolor}

	\definecolor{azure}{rgb}{0.3, 0.4, 0.86}
    
\usepackage{hyperref} 
\hypersetup{
    colorlinks,
    linkcolor=azure,
    citecolor=azure,
    urlcolor=azure
	}

\usepackage{amsmath}
\usepackage{amssymb}
\usepackage{mathrsfs}
\usepackage{amscd}
\usepackage{graphicx}
\usepackage{mathdots}
\usepackage[normalem]{ulem}

\usepackage{gensymb}
\usepackage{epstopdf}
 \usepackage{todonotes}
 \usepackage{youngtab}
 \usepackage{wrapfig}
 \usepackage{float}

 \usepackage{bbm}
 \usepackage{bm}
 \usepackage{verbatim}

 \usepackage[nocompress]{cite}
 \usepackage{epsfig}       %For postscript
\usepackage{epic,eepic}        %For epic and eepic output from xfig

\usepackage{tikz}
\usetikzlibrary{arrows.meta}

\usepackage{palatino}
\usepackage[margin=0.82in]{geometry}
 
\newtheorem{thm}{Theorem}[section]
\newtheorem{prop}[thm]{Proposition}

\newtheorem{lem}[thm]{Lemma}
\newtheorem{cor}[thm]{Corollary}
\newtheorem{conj}[thm]{Conjecture}

\theoremstyle{definition}

\newtheorem{problem}[thm]{Problem}
\newtheorem{question}[thm]{Question}

\newcommand\norm[1]{\left\lVert#1\right\rVert}

\newtheorem{remark}[thm]{Remark}

\numberwithin{equation}{section}

\begin{document}

\title{On creating convexity in high dimensions}

\author{Samuel G. G. Johnston}
\address{Department of Mathematics, King's College London}
\email{samuel.g.johnston@kcl.ac.uk}

\begin{abstract}
Given a subset $A$ of $\mathbb{R}^n$, we define
\begin{align*}
\mathrm{conv}_k(A) := \left\{ \lambda_1 s_1 + \cdots + \lambda_k s_k : \lambda_i \in [0,1], \sum_{i=1}^k \lambda_i = 1 , s_i \in A \right\}
\end{align*}
to be the set of vectors in $\mathbb{R}^n$ that can be written as a $k$-fold convex combination of vectors in $A$.
Let $\gamma_n$ denote the standard Gaussian measure on $\mathbb{R}^n$. We show that for every $\varepsilon > 0$, there exists a subset $A$ of  $\mathbb{R}^n$ with Gaussian measure $\gamma_n(A) \geq 1- \varepsilon$ such that for all $k = O_\varepsilon(\sqrt{\log \log(n)})$, $\mathrm{conv}_k(A)$ contains no convex set $K$ of Gaussian measure $\gamma_n(K) \geq \varepsilon$.
This result acts as a complement to the recent affirmative resolution of Talagrand's convexity conjecture by Hua, Song, and Tudose, which states that a universal dilation of the threefold Minkowski sum $A+A+A$ of a large set $A$ guarantees a large convex subset. Our approach utilises concentration properties of random copulas and the application of optimal transport techniques to the empirical coordinate measures of vectors in high dimensions.
\end{abstract}

\maketitle
\section{Introduction and overview}
\subsection{Background}
A subset $K$ of $\mathbb{R}^n$ is said to be convex if the convex combination $\lambda s + (1-\lambda)s'$ lies in $K$ for all $s,s' \in K$ and all $\lambda \in [0,1]$.
The convex hull $\mathrm{conv}(A)$ of a subset $A$ of $\mathbb{R}^n$ is the smallest convex set containing $A$.
Carathéodory's theorem in convex geometry states that the convex hull of $A$ consists precisely of the set of points in $\mathbb{R}^n$ that can be written as a convex combination of $n+1$ points in $A$. 

This article is motivated by a problem raised by Talagrand concerning whether, given a large set $A$, a reasonably large convex subset of $\mathbb{R}^n$ can be constructed from elements of $A$ using a fixed number $k$ of operations --- in such a way that $k$ is independent of the underlying dimension $n$. In \cite{talaPACM}, Talagrand asks the following:
\begin{quote} 
How many operations are required to build the convex hull of $A$ from $A$? Of course the exact answer should depend on what exactly we call  ``operation'',
but Carathéodory's theorem asserts that any point in the convex hull of $A$ is in the convex hull of a subset $B$ of $A$
such that $|B| = n+1$ (and this cannot be improved) so
we should expect that the number of operations is of order
$n$, and that this cannot really be improved.

Suppose now that, in some sense, $A$ is ``large'', and that,
rather than wanting to construct all the convex hull of $A$,
we only try to construct ``a proportion of it''. Do we really
need a number of operations that grows with $n$?
\end{quote}

To give a precise formulation of this question, we need a notion of \emph{large} and a notion of \emph{operation}. With a view to characterising the first, let $\gamma_n$ be the standard Gaussian measure on $\mathbb{R}^n$, so that the Gaussian measure of a subset $A$ of $\mathbb{R}^n$ is given by 
\begin{align*}
\gamma_n(A) := \int_A (2\pi)^{-n/2} e^{ -\norm{s}_2^2/2}\mathrm{d}s.
\end{align*}
Up to dilations, the standard Gaussian measure is the unique probability measure on $\mathbb{R}^n$ that is invariant under rotations and has the property that a random vector distributed according to this probability measure has independent coordinates. Thus the Gaussian measure $\gamma_n(A)$ of a set $A$ provides a natural notion of its size.

We will consider two operations on a subset $A$ of $\mathbb{R}^n$: $k$-fold Minkowski addition and $k$-fold convex combinations. These are given by
\begin{align} \label{eq:mink}
A^{\oplus k} &:= \{ s_1 + \cdots + s_k : s_i \in A \},\\
\mathrm{conv}_k(A) &:= \left\{ \lambda_1 s_1 + \cdots + \lambda_k s_k :
\lambda_i \in [0,1], \sum_{i=1}^k \lambda_i = 1,\, s_i \in A \right\}.
\end{align}
That is, $A^{\oplus k}$ is the set of sums of $k$ vectors in $A$, while
$\mathrm{conv}_k(A)$ is the set of convex combinations of at most $k$ vectors in $A$.

Observe that for any subset $A$ of $\mathbb{R}^n$, $\mathrm{conv}_1(A)=A$, $\mathrm{conv}_k(A) \subseteq \mathrm{conv}_{k+1}(A)$, and $A$ is convex if and only if $\mathrm{conv}_2(A)=  A$. Moreover, Carathéodory's theorem says that $\mathrm{conv}_{n+1}(A) = \mathrm{conv}(A)$. 

We say that a subset $A$ of $\mathbb{R}^n$ is balanced if $a \in A, \lambda \in [-1,1]$ implies $\lambda a \in A$. The $k$-fold Minkowski sum of a balanced set $A$ consists of all vectors of the form $\lambda_1 s_1 + \cdots + \lambda_k s_k$ with $s_i \in A$ and with $\lambda_i$ any real numbers satisfying $\sum_{i=1}^k |\lambda_i| \leq k$. In particular, if $A$ is balanced, then $A^{\oplus k}$ is also balanced, and $A^{\oplus k} \subseteq A^{\oplus (k+1)}$. Crucially, if $A$ is balanced, then $A^{\oplus k}$ is equal to the dilation $A^{\oplus k} = k\mathrm{conv}_k(A) = \{ k s : s \in \mathrm{conv}_k(A)\}$ of $\mathrm{conv}_k(A)$. In particular, when $A$ is balanced we have $\mathrm{conv}_k(A) \subsetneq A^{\oplus k}$, though of course the former set is significantly smaller than the latter. 
Thus for balanced sets, Minkowski summation may be viewed as a combination of two distinct effects: repeated convexification and a simultaneous dilation by a factor $k$.

With these notions at hand, in the 1990s Talagrand posed the following conjecture about large convex subsets of Minkowski sums:

\begin{conj}[Talagrand's convexity conjecture \cite{talagafa, talaPACM, talaproblems}] \label{conj:tala}
There exists $\varepsilon > 0$ and an integer $k \geq 2$ such that for every $n \geq 1$ and every balanced subset $A$ of $\mathbb{R}^n$ with $\gamma_n(A) \geq 1-\varepsilon$, the dilated Minkowski sum $\varepsilon^{-1} A^{\oplus k}$ contains a convex subset $K$ with $\gamma_n(K) \geq \varepsilon$. 
\end{conj}

In other words, Conjecture \ref{conj:tala} states that for a sufficiently large integer $k$, the $k$-fold Minkowski sum of any large balanced set necessarily contains a reasonably large convex subset, regardless of the underlying dimension. 

This conjecture is discussed at length in \cite{talagafa} and \cite{talaPACM}, and more recently in \cite{talaproblems}. In \cite{talagafa}, an intricate probabilistic argument involving Gaussian comparison techniques \cite{talaacta} is used to show that $k = 2$ is certainly not sufficient. Namely, it is shown there that for any $\varepsilon > 0$ it is possible to construct a subset $A$ of $\mathbb{R}^n$ of Gaussian measure $\gamma_n(A) \geq 1- \varepsilon$ such that $\varepsilon^{-1}(A+A) := \{ \varepsilon^{-1}(s_1+s_2) : s_i \in A\}$ does not contain a convex subset $K$ of Gaussian measure $\gamma_n(K) \geq \varepsilon$. Alongside Conjecture \ref{conj:tala}, Talagrand \cite{talaPACM} also proposed several related combinatorial conjectures motivated by these questions, which have received considerable attention and have largely been resolved \cite{FKNP,PP,PP2}.

Recalling that for balanced sets $A$, $\mathrm{conv}_k(A) \subsetneq A^{\oplus k}$ and $A^{\oplus k} = k \mathrm{conv}_k(A)$, we address a question stronger than Talagrand's original formulation:

\begin{question}
Does there exist an $\varepsilon > 0$ and an integer $k \geq 2$ such that for every $n \geq 1$ and every balanced subset $A$ of $\mathbb{R}^n$ with $\gamma_n(A) \geq 1-\varepsilon$, $\mathrm{conv}_k(A)$ contains a convex subset $K$ with $\gamma_n(K) \geq \varepsilon$? 
\end{question}

We use tools from optimal transport to prove a negative answer to this question, asserting that even when $k = O_\varepsilon(\sqrt{\log \log(n)})$, there are large subsets $A$ of $\mathbb{R}^n$ for which $\mathrm{conv}_k(A)$ contains no convex subset of reasonable size:

\begin{thm} \label{thm:main}
Let $n \geq 1$ and $\varepsilon \in (1/n,1/2]$. Then there exists a balanced subset $A_n$ of $\mathbb{R}^n$ with size $\gamma_n(A_n) \geq 1-C/n$, and such that for every integer $k$ satisfying
\begin{align*}
1 \leq k \leq c \sqrt{\log \log(n)- \log \log(1/\varepsilon)} - C,
\end{align*}
the set $\mathrm{conv}_k(A_n)$ contains no convex subset $K$ satisfying $\gamma_n(K) \geq \varepsilon$.
\end{thm}
The constants $c,C>0$ in Theorem \ref{thm:main} are universal.

\medskip
Since this article first appeared online, Talagrand's original question, Conjecture \ref{conj:tala}, has been resolved in the affirmative in works by Song \cite{song} and Hua, Song, and Tudose \cite{HST}. In \cite{song}, Conjecture \ref{conj:tala} was shown to be equivalent to a statement involving representing subgaussian random vectors as sums of dependent Gaussian vectors. In the subsequent work \cite{HST}, this representation conjecture was proved true, yielding the consequential geometric result:

\begin{thm}[Hua, Song and Tudose \cite{HST}] \label{thm:HST}
There exists $\varepsilon > 0$ such that for all $n \geq 1$ and every closed subset $A$ of $\mathbb{R}^n$ with $\gamma_n(A) \geq 1- \varepsilon$, the dilated threefold Minkowski sum $\varepsilon^{-1}(A+A+A)$ contains a convex subset $K$ of Gaussian measure $\gamma_n(K) \geq 1/2$. 
\end{thm}

Consequently, Talagrand's convexity conjecture holds, and in fact may be taken with $k=3$.

Our main result, Theorem \ref{thm:main}, acts as a complement to the breakthrough in Theorem \ref{thm:HST}:
the dilated threefold Minkowski sum $\varepsilon^{-1}(A+A+A)$ of any large set $A$ is guaranteed to contain a reasonable convex subset, but there is no universal constant $k$ for which the same can be said about $\mathrm{conv}_k(A)$. This establishes that bounded convexification behaves fundamentally differently from bounded Minkowski summation in high dimensions. In particular, the convexity creation phenomenon established by Hua, Song, and Tudose \cite{HST} cannot be explained solely in terms of taking bounded convex combinations; rather, the enlargement of sets due to dilation and Minkowski addition is an essential mechanism driving the creation of large convex subsets.

%%%%%%%%%%%%%%%%%%%%%%%%%%%%%%%%%%%
\subsection{Overview of our proof}

Our proof of Theorem \ref{thm:main}, like the recent works of Song \cite{song} and Hua, Song, and Tudose \cite{HST}, is based on translating geometric operations on high-dimensional vectors into probabilistic questions about sums of random variables.

From the perspective of Gaussian measure, the coordinates of a typical high-dimensional vector are approximately Gaussian. This allows geometric operations on vectors to be recast as probabilistic questions about Gaussian random variables. The essence of this viewpoint is already present in Talagrand's original work \cite{talagafa}: the coordinates of a sum or convex combination of vectors with approximately Gaussian coordinates behave like sums or convex combinations of dependent Gaussian random variables.

While this observation underlies both projects, the tools and conclusions are quite different. The works of Song \cite{song} and Hua, Song, and Tudose \cite{HST} show that by taking scaled sums $\varepsilon^{-1}(s_1+s_2+s_3)$ of vectors $s_1,s_2,s_3\in\mathbb{R}^n$ with approximately Gaussian coordinates, one can generate a sufficiently rich collection of coordinate distributions to guarantee the emergence of large convex subsets. Specifically, their approach translates questions about high-dimensional vectors into questions about the laws of subgaussian random vectors in fixed dimensions.

In contrast, we use optimal transport techniques to show that the coordinates of a convex combination $\lambda_1s_1+\cdots+\lambda_ks_k$ of vectors $s_1,\ldots,s_k\in\mathbb{R}^n$ with approximately Gaussian coordinates satisfy certain quantitative restrictions. These restrictions are modest at the level of one-dimensional distributions, but become decisive when combined with concentration of measure in high dimensions: any large convex set must contain vectors whose empirical coordinate measures are close to the relevant extremal threshold. Our approach reduces convex combinations of high-dimensional vectors to the laws of one-dimensional random variables, allowing tools from optimal transport to be brought to bear. 
% We show that bounded convex combinations cannot push the associated empirical coordinate measures far enough to meet this threshold, and hence cannot force the emergence of large convex subsets.

\vspace{3mm}
To make these ideas more concrete, let $W_n := \{ t = (t^1,\ldots,t^n) \in \mathbb{R}^n : t^1 \leq \cdots \leq t^n \}$ be the set of vectors in $\mathbb{R}^n$ with nondecreasing coordinates. Let $\mathcal{W}_n$ be the set of probability measures on the real line that can be written as a sum of $n$ Dirac masses of size $1/n$. The map sending $t \in W_n$ to the probability measure $\frac{1}{n} \sum_{j=1}^n \delta_{t^j}$ in $\mathcal{W}_n$ is bijective. Let $\mathcal{S}_n$ denote the symmetric group of degree $n$. Consider the map 
\begin{align} \label{eq:correspondence}
\mathbb{R}^n \mapsto (\mathcal{W}_n,\mathcal{S}_n) \qquad s \mapsto (\mu_s,\sigma_s),
\end{align}
where for a vector $s = (s^1,\ldots,s^n)$ in $\mathbb{R}^n$, $\mu_s$ is the empirical measure of the coordinates of $s$ and $\sigma_s$ is the permutation that sorts the coordinates of $s$ in nondecreasing order. That is:
\begin{align*}
\mu_s := \frac{1}{n}\sum_{j=1}^n \delta_{s^j} \qquad \text{and} \qquad s^{\sigma_s(1)} \leq \cdots \leq s^{\sigma_s(n)}.
\end{align*}
If $s$ has two coordinates the same, i.e.\ if $s^{j_1} = s^{j_2}$ with $j_1 < j_2$ say, we break ties by preserving their original coordinate order. We call $\mu_s$ the empirical coordinate measure of $s$ and $\sigma_s$ the coordinate ordering permutation of $s$. 
The central idea of our approach hinges on an understanding of how taking convex combinations of vectors interacts with the correspondence in \eqref{eq:correspondence}.
\vspace{2mm}

When $n$ is large, the overwhelming majority of $\mathbb{R}^n$ is comprised of vectors $s$ for which $\mu_s$ is close to the standard one-dimensional Gaussian distribution $\gamma$. To formalise this idea, let $W(\mu_s,\gamma)$ denote the ($L^1$-)Wasserstein distance between $\mu_s$ and $\gamma$ (see Section \ref{sec:Wasser} for the definition), and define
\begin{align} \label{eq:Efirst}
E_n(\delta) := \{ s \in \mathbb{R}^n : W(\mu_s,\gamma) \leq \delta \}.
\end{align}
to be the set of vectors $s$ in $\mathbb{R}^n$ whose empirical coordinate measure $\mu_s$ lies within Wasserstein distance $\delta$ of the standard Gaussian law.

The following result, which we derive fairly quickly in Section \ref{sec:DKW} from the Dvoretzky-Kiefer-Wolfowitz inequality \cite{DKW} for Kolmogorov-Smirnov distances, states that in terms of Gaussian measure, most of high-dimensional space consists of vectors whose empirical coordinate measures are nearly Gaussian:
\begin{prop} \label{prop:DKW}
Let $\delta_n = C n^{-1/3}$. Then $A_n := E_n(\delta_n)$ has overwhelmingly large Gaussian measure in that $\gamma_n(A_n) \geq 1-C/n$. Here $C > 0$ is a universal constant. 
\end{prop}

The rate $n^{-1/3}$ in Proposition~\ref{prop:DKW} is chosen for simplicity and is not optimal, as any polynomial rate of decay would suffice for our subsequent arguments.

In our proof of Theorem \ref{thm:main} we take the balanced version $\tilde{A}_n = [-1,1]A_n := \{ \lambda s : s \in A_n, \lambda \in [-1,1]\}$ of the set $A_n$ occurring in the statement of Proposition \ref{prop:DKW}. 

Although $A_n$ is defined through a particular Wasserstein condition, it is not exceptional from the perspective of Gaussian measure. Indeed, any subset $B$ of $\mathbb{R}^n$ satisfying $\gamma_n(B) \geq 1-\varepsilon$ consists almost entirely of vectors in $A_n$, since
$\gamma_n(A_n \cap B) \geq 1-\varepsilon-Cn^{-1/3}.
$
Thus any set of large Gaussian measure must overlap $A_n$ in all but a negligible proportion of its mass.

\medskip
To formalise the probabilistic behaviour of convex combinations of almost-Gaussian vectors, we first characterize the probability distributions that can arise from arbitrary dependent sums of Gaussians. Let $\gamma = \gamma_1$ denote the standard Gaussian law, and for $k \geq 1$ define 
\begin{align} \label{eq:Mkdef}
\mathcal{M}_k := \{ \mu : \text{$\mu$ is the law of } \lambda_1Z_1+ \cdots +\lambda_kZ_k \text{ where $Z_i \sim \gamma$} \}
\end{align}
to be the set of probability laws $\mu$ on the real line that occur as the law of a convex combination of $k$ standard Gaussian random variables, coupled in any way possible. Note in particular that $\mathcal{M}_1$ consists only of the standard one-dimensional Gaussian law itself, i.e.\ $\mathcal{M}_1 = \{ \gamma \}$.

If we define $B(\mu, \delta):= \{ \nu \text{ prob.\ measure on $\mathbb{R}$}: W(\mu,\nu) \leq \delta \}$ to be the closed ball of radius $\delta$ around a probability measure $\mu$ on $\mathbb{R}$ in the Wasserstein metric, then the set $E_n(\delta)$ defined in \eqref{eq:Efirst} may alternatively be written
\begin{align} \label{eq:Edef}
E_n(\delta) := \{ s \in \mathbb{R}^n : \mu_s \in B(\gamma,\delta) \}.
\end{align}

In the sequel we establish the following result, which is a stability property for the representation \eqref{eq:Edef} following from natural properties enjoyed by Wasserstein distances under couplings:
\begin{prop} \label{prop:comma}
We have
\begin{align} \label{eq:Edef2}
\mathrm{conv}_k(E_n(\delta)) \subseteq \{ s \in \mathbb{R}^n : \mu_s \in B(\mathcal{M}_k,\delta) \},
\end{align}
where $B(\mathcal{M}_k,\delta ) := \bigcup_{ \mu \in \mathcal{M}_k } B(\mu,\delta)$. 
\end{prop}

In other words, if $E_n(\delta)$ consists of vectors whose empirical coordinates lie within $\delta$ of standard Gaussian, then $\mathrm{conv}_k(E_n(\delta))$ consists only of vectors whose empirical coordinate measures lie within $\delta$ of the law of a $k$-fold convex combination of standard Gaussians.

Consider now in particular taking convex combinations of vectors in the set $A_n := E_n(\delta_n)$ defined in the statement of Proposition \ref{prop:DKW}. Proposition \ref{prop:comma} states that the empirical coordinate measure of an element $s \in \mathrm{conv}_k(A_n)$ must be within Wasserstein distance $\delta_n$ of some $\mu \in \mathcal{M}_k$. When $n$ is large compared to $k$, this condition places a significant restriction on $\mathrm{conv}_k(A_n)$. 

To characterise this restriction, we introduce a simple functional on probability measures which supplies a means of distinguishing between measures that are close to some element of $\mathcal{M}_k$ and those that are not. Namely, we simply look at the mass the measure gives to $[1,\infty)$:
\begin{align*}
\mathrm{Exc}_1(\mu) := \mu([1,\infty)).
\end{align*}
We call $\mathrm{Exc}_1(\mu)$ the exceedance of the measure $\mu$.  We will also refer to the exceedance of a vector $s \in \mathbb{R}^n$ to be the exceedance $\mathrm{Exc}_1(\mu_s)$ of its empirical coordinate measure, which simply captures the fraction of coordinates exceeding $1$:
\begin{align*}
\mathrm{Exc}_1(\mu_s) = \frac{1}{n} \# \{ 1 \leq i \leq n : s_i \geq 1 \}.
\end{align*}

We are led to consider the following problem in optimal transport:

\begin{problem} \label{prob:Sk}
Calculate the supremum $S_k$ of $\mathrm{Exc}_1(\mu)$ over all measures $\mu$ in $\mathcal{M}_k$. That is, calculate
\begin{align*}
S_k := \sup_{ \Pi } \sup_{ \lambda_1,\ldots,\lambda_k } \mathbf{P} \left( \lambda_1 Z_1 + \ldots + \lambda_k Z_k \geq 1 \right),
\end{align*}
where the supremum is taken over all couplings $\Pi$ of random variables $(Z_1,\ldots,Z_k)$ with standard Gaussian marginals, and over all $\lambda_1,\ldots,\lambda_k \in [0,1]$ satisfying $\sum_{i=1}^k \lambda_i = 1$.
\end{problem}

We are particularly interested in the asymptotic behaviour of $S_k$ as $k \to \infty$. As a starting point, it is possible to show using a first moment argument that for every $k \geq 1$ we have the following absolute bound
\begin{align} \label{eq:universal}
S_k \leq p_1
\end{align}
where $p_1$ is the unique solution $p \in (0,1)$ to the equation $\mathbf{E}[ Z | Z > \Phi^{-1}(1-p) ] = 1$; see Lemma \ref{lem:firstmoment}. 

We are able to refine the bound in \eqref{eq:universal} using a powerful idea in 
optimal transport called Monge-Kantorovich duality. Given a measurable functional 
$c:\mathbb{R}^k \to [0,\infty)$ and a $k$-tuple of probability measures $\mu_1,\ldots,\mu_k$, the central problem in optimal transport is to calculate or estimate the supremum $\sup_\Pi \mathbf{E}_\Pi [ c(Z_1,\ldots,Z_k)]$, where the supremum runs over all couplings $\Pi$ of the probability measures $\mu_1,\ldots,\mu_k$. The celebrated Monge-Kantorovich duality states that under mild conditions we have the relation
\begin{align} \label{eq:MKduality00}
\sup_\Pi \mathbf{E}_\Pi \left[ c(Z_1,\ldots,Z_k) \right] = \inf_{f_1,\ldots,f_k} \sum_{i=1}^k \mathbf{E}_{\mu_i} [ f_i(Z_i) ]
\end{align}
where the infimum is taken over all $f_1,\ldots,f_k$ satisfying $\sum_{i=1}^k f_i(z_i) \geq c(z_1,\ldots,z_k)$. See e.g., Villani \cite{villani}.

By letting $c(z_1,\ldots,z_k) = \mathrm{1} \{ \lambda_1 Z_1 + \cdots + \lambda_k Z_k \geq 1 \}$ and then subsequently optimising over the functions $f_1,\ldots,f_k$, we are able to refine the universal bound in \eqref{eq:universal} to obtain the following $k$-dependent bound:

\begin{thm} \label{thm:MK}
There are universal constants $c,C>0$ such that 
\begin{align} \label{eq:MK0}
S_k \leq p_1 - c e^{ - Ck^2}.
\end{align}
\end{thm}

We note in particular, for each fixed $k$, there does not exist a measure $\mu$ in $\mathcal{M}_k$ such that $\mathrm{Exc}_1(\mu) = p_1$.

By exploiting continuity properties of the exceedance functional in Wasserstein distances in conjunction with Proposition \ref{prop:comma}, we are able to obtain as a corollary the following result which is an upper bound on the possible exceedance of an element of $\mathbb{R}^n$ that can be written as a convex combination of $k$ elements of $E_n(\delta)$. 
\begin{thm} \label{thm:upper}
Let $s \in \mathrm{conv}_k(E_n(\delta))$. Then for $0 \leq \delta \leq 1/2$ and $k \geq 1$ we have 
\begin{align*}
\mathrm{Exc}_1(\mu_s) \leq p_1 - c e^{ - Ck^2} + C \sqrt{\delta}.
\end{align*}
\end{thm}

Thus provided $\delta$ is small compared to $ce^{ - Ck^2}$, the exceedances of vectors in $\mathrm{conv}_k(E_n(\delta))$ are bounded away from the theoretical threshold $p_1$. 

We will apply Theorem \ref{thm:upper} with $\delta_n = C n^{-1/3}$, and conclude by Proposition \ref{prop:DKW} that $A_n = E_n(\delta_n)$ is a large subset of $\mathbb{R}^n$ with the property that the exceedances of vectors in $\mathrm{conv}_k(A_n)$ are bounded away from $p_1$ by the upper bound $\mathrm{Exc}_1(\mu_s) \leq p_1 - ce^{ - Ck^2} + C n^{-1/6}$. When $n$ is large compared to $k$, this bound places a strong constraint on the elements of $\mathrm{conv}_k(A_n)$; we will see below that any reasonably large convex set $K$ necessarily contains a vector not satisfying this constraint.

\vspace{3mm}
Theorem \ref{thm:upper} is one half of the proof of our main result, Theorem \ref{thm:main}. The other (more difficult) half of the proof lies in showing that 
\emph{any} sufficiently large convex set $K$ in high-dimensional space necessarily contains vectors whose empirical coordinate measures have exceedances arbitrarily close to $p_1$. Our first step in this direction is the following reverse inequality to Theorem \ref{thm:MK}:
\begin{align} \label{eq:lowerness}
S_d \geq p_1 - C \frac{\log d}{\sqrt{d}},
\end{align}
where we introduce a new integer $d \geq 1$ that will play a different role to the integer $k$ occurring in the statement of Theorem \ref{thm:main}. 

We prove \eqref{eq:lowerness} by considering a simple coupling that we call the box-product coupling at $q$. 
This is the coupling $\Pi$ of $d$ standard Gaussian random variables so that they all exceed the value $q$ at precisely the same time, but are otherwise independent, so that the probability density function on $\mathbb{R}^d$ of this coupling is given by
\begin{align*}
\Pi(\mathrm{d}x) := \left( (1-p)^{-(d-1)} \mathrm{1}_{\{ x_i < q ~\forall i=1,\ldots,d \}} + p^{-(d-1)} \mathrm{1}_{\{ x_i \geq q ~\forall i=1,\ldots,d \}}   \right) \gamma_d(\mathrm{d}x),
\end{align*}
where $1-p := \Phi(q)$ and $\gamma_d(\mathrm{d}x)$ is the standard $d$-dimensional Gaussian density. 
For a carefully chosen value of $q$ of the form $\Phi^{-1}(1 - p_1 ) + O(\sqrt{\log(d)/d})$, it transpires that the law $\mu_d \in \mathcal{M}_d$ of $d^{-1}(Z_1+\ldots+Z_d)$ has an exceedance of at least $p_1 - C\sqrt{\log d/d}$, thereby establishing \eqref{eq:lowerness}.
See Corollary \ref{cor:underbound} for further details. 

\vspace{3mm}
The most intricate part of the proof of our main result is showing that when $n$ is large, any convex set $K$ of Gaussian measure $\gamma_n(K)\geq \varepsilon$ necessarily contains $d$ distinct vectors $s_1,\ldots,s_d$ such that the $d$-dimensional empirical coordinate measure
\begin{align*}
\mu_{s_1,\ldots,s_d} := \frac{1}{n} \sum_{j=1}^n \delta_{(s_1^j,\ldots,s_d^j)}
\end{align*}
of the $d$-tuple $(s_1,\ldots,s_d)$ (which is a measure on $\mathbb{R}^d$) is close to a box-product coupling of Gaussians. To give a brief idea of how we establish that such a $d$-tuple exists here, given a vector $t=(t^1,\ldots,t^n) \in \mathbb{R}^n$ and a permutation $\sigma \in \mathcal{S}_n$, let $\sigma t = (t^{\sigma(1)},\ldots,t^{\sigma(n)})$. 
Suppose that $K$ (convex or otherwise) has Gaussian measure $\gamma_n(K) \geq \varepsilon$. Then by Proposition \ref{prop:DKW}, $\gamma_n(K \cap A_n) \geq \varepsilon/2$ (provided $C/n \leq \varepsilon/2$). Now for $t \in \mathbb{R}^n$, let
\begin{align*}
N(t) := \# \{ \sigma \in \mathcal{S}_n : \sigma t \in K \cap A_n \}
\end{align*}
be the number of ways of reordering the coordinates of $t$ to obtain a vector in $K \cap A_n$. By symmetry we have
\begin{align*}
\int_{\mathbb{R}^n} N(t) \gamma_n(\mathrm{d}t) = n! \gamma_n( K \cap A_n) \geq (\varepsilon/2)n!.
\end{align*}
In particular, there exists some $t \in K \cap A_n$ such that $N(t) \geq (\varepsilon/2)n!$. 

We then establish using a probabilistic argument that since the collection of ways of reordering the coordinates of $t$ to obtain an element of $K \cap A_n$ is sufficiently rich, when $n$ is sufficiently large compared to $d$ there must be a $d$-tuple of permutations $(\sigma_1,\ldots,\sigma_d)$ such that each $\sigma_it$ lies in $K$, and moreover such that $\mu_{\sigma_1 t,\ldots,\sigma_d t}$ is close to a box-product coupling. This argument appeals to concentration properties of random permutons \cite{AJ, HKMRS, KKRW}. 

If $K$ is convex, then the convex combination $u := d^{-1}( \sigma_1 t + \ldots + \sigma_d t)$ of vectors also lies in $K$ and has coordinate empirical distribution close to $\mu_d$, which has a high exceedance. Putting these ideas together, we are able to prove the following result:

\begin{thm} \label{thm:lower} 
Let $\varepsilon \in (0,1/2]$ and $n \in \mathbb{N}$. Let $K$ be a convex subset of $\mathbb{R}^n$ with $\gamma_n(K) \geq \varepsilon$. Then $K$ contains a vector $u$ whose exceedance satisfies
\begin{align*}
\mathrm{Exc}_1(\mu_u) \geq p_1 - C_\varepsilon \log(n)^{-1/3},
\end{align*}
where $C_\varepsilon = C \log(1/\varepsilon)^{1/3}$ for some universal $C > 0$. 

\end{thm}

We now spell out how Theorem \ref{thm:main} follows from Proposition \ref{prop:DKW}, Theorem \ref{thm:upper}, and Theorem \ref{thm:lower}:

\begin{proof}[Proof of Theorem \ref{thm:main} assuming Proposition \ref{prop:DKW}, Theorem \ref{thm:upper} and Theorem \ref{thm:lower}]
By Proposition \ref{prop:DKW}, if $\delta_n = C n^{-1/3}$, the set $A_n := E_n(\delta_n) \subseteq \mathbb{R}^n$ has Gaussian measure at least $1 - C/n$. Now let 
\begin{align*}
\tilde{A}_n := [-1,1]A_n := \{ \lambda s : s \in A_n, \lambda \in [-1,1] \}.
\end{align*}
Then $\tilde{A}_n$ is balanced and has Gaussian measure at least as large as that of $A_n$. Note further that 
\begin{align*}
\mathrm{conv}_k(\tilde{A}_n) = \{ \lambda s : s \in \mathrm{conv}_k(A_n), \lambda \in [-1,1] \},
\end{align*}
which implies that
\begin{align*}
\sup_{s \in \mathrm{conv}_k(\tilde{A}_n)} \mathrm{Exc}_1(\mu_s) = \sup_{s \in \mathrm{conv}_k(A_n)} \mathrm{Exc}_1(\mu_s).
\end{align*}
Now on the one hand, by setting $\delta_n = C n^{-1/3}$ in Theorem \ref{thm:upper}, we see that the exceedance $\mathrm{Exc}(\mu_s)$ associated with any vector $s$ in $\mathrm{conv}_k(A_n)$ must satisfy
\begin{align*}
\mathrm{Exc}_1(\mu_s) \leq p_1 - c e^{ - Ck^2} + C n^{-1/6}.
\end{align*}
Conversely, if $K$ is a convex subset of Gaussian measure at least $\varepsilon$, then by Theorem \ref{thm:lower}, $K$ contains a vector $u$ whose exceedance satisfies
\begin{align*}
\mathrm{Exc}_1(\mu_u) \geq p_1 - C_\varepsilon \log(n)^{-1/3}.
\end{align*}
where $C_\varepsilon = C \log(1/\varepsilon)^{1/3}$. 
It follows that if 
\begin{align} \label{eq:kk}
C_\varepsilon \log(n)^{-1/3} < c e^{ - Ck^2} - C n^{-1/6},
\end{align}
then $K$ cannot be contained in $\mathrm{conv}_k(A_n)$. 

Unraveling the inequality \eqref{eq:kk} to make $k$ the subject, we see that this is equivalent to
\begin{align*}
k \leq c \sqrt{ \log \log(n) - \log \log(1/\varepsilon)} - C,
\end{align*}
thereby completing the proof of Theorem \ref{thm:main}.

\end{proof}

We conclude our proof overview by noting a potential avenue for improving the rate in Theorem \ref{thm:main}. There is a large gap in the lower and upper bounds we obtain in the inequality
\begin{align} \label{eq:Skb}
p_1 - C e^{ - c k^2 } \geq S_k \geq p_1 - C \frac{ \log k}{ \sqrt{k}},
\end{align}
which follows from Theorem \ref{thm:MK} and \eqref{eq:lowerness}. If we were able to sharpen these bounds, it is likely one could improve on the $O(\sqrt{\log \log(n)})$ rate in the statement of Theorem \ref{thm:main}. 

The author conjectures that the upper bound in \eqref{eq:Skb} is closer to the true value of $S_k$ than the lower bound, and believes in particular that there is likely scope to improve on the lower bound by using a more sophisticated coupling. 
We should say however that the box-product coupling was chosen for its relatively simple structure, which facilitates the (already quite delicate) probabilistic proof of Theorem \ref{thm:lower}. Although a more sophisticated coupling might yield a sharper bound, it would likely come at the cost of significantly increased complexity in the argument. 

\subsection{Further discussion and related work}

As discussed in the introduction, the recent resolution of Talagrand's conjecture by Hua, Song, and Tudose \cite{HST} relies on the representation of subgaussian random vectors \cite{song}. In a related development, van Handel has recently proved a strengthening of Talagrand's subgaussian comparison theorem \cite{vanhandel}.

To frame the geometric implications of our result, recall that the Carathéodory number of a subset $A$ of $\mathbb{R}^n$ is the smallest integer $k$ such that $\mathrm{conv}_k(A)=\mathrm{conv}(A)$. Upper bounds on Carathéodory numbers for certain classes of sets have been established in \cite{8,9}. If we define the \emph{effective Carathéodory number} of $A$ to be the smallest integer $k$ such that $\mathrm{conv}_k(A)$ contains a convex subset of Gaussian measure at least $\frac{1}{10}\gamma_n(A)$, then Theorem \ref{thm:main} implies that there exist arbitrarily large subsets $A$ of $\mathbb{R}^n$ whose effective Carathéodory numbers are at least $O(\sqrt{\log\log n})$.

On a related note, an argument due to Maurey (see also \cite{pisier}) shows that every point in the convex hull $\mathrm{conv}(A)$ of a subset $A$ of the Euclidean unit ball $B(0,1)\subset\mathbb{R}^n$ lies within distance $1/\sqrt{k}$ of a $k$-fold convex combination of points in $A$, irrespective of the ambient dimension; see \cite[Theorem 0.0.2]{vershynin}.

Our proof of Theorem \ref{thm:main} uses the idea of associating $d$-tuples of permutations in the symmetric group $\mathcal{S}_n$ with copulas on $[0,1]^d$. Namely, each $d$-tuple of permutations $(\sigma_1,\ldots,\sigma_d)$ in $\mathcal{S}_n$ may be associated with a probability measure on $[0,1]^d$ supported on hypercubes of side length $1/n$; see \eqref{eq:cd}. This idea first appeared for $d=2$ in \cite{GGK,HKMRS}, and for general $d$ in \cite{CHS}. A large deviation principle for the limiting copula was proved for $d=2$ in \cite{KKRW}. The author and Octavio Arizmendi \cite{AJ} use a higher-dimensional analogue of this large deviation principle to develop optimal transport formulations of operations in free probability. A key ingredient in the proof of Theorem \ref{thm:lower} is the concentration of the probability measure associated with a $d$-tuple of independent uniform permutations around Lebesgue measure on $[0,1]^d$.

The broader idea of associating high-dimensional vectors with probability measures through their empirical coordinate distributions and studying continuity properties in Wasserstein distance also appears in recent work of the author and McSwiggen \cite{JM} on an asymptotic version of Horn's problem.

%%%%%%%%%%%%%%%%%%%%%%%%%%%%%%%%%%%%%%%%%%%%%%%%%%%%%%%%%%%%
\subsection{Overview}
%%%%%%%%%%%%%%%%%%%%%%%%%%%%%%%%%%%%%%%%%%%%%%%%%%%%%%%%%%%%
As outlined in the introduction, the proof of our main result, Theorem \ref{thm:main}, follows from Proposition \ref{prop:DKW}, Theorem \ref{thm:upper} and Theorem \ref{thm:lower}. As such, the remainder of the article is structured as follows:

\begin{itemize}
\item In Section \ref{sec:OT} we discuss couplings and some basic notions from optimal transport, and prove Proposition \ref{prop:comma}.
\item In Section \ref{sec:DKW} we prove Proposition \ref{prop:DKW}.
\item In Section \ref{sec:upper} we study couplings of Gaussian random variables that maximise probabilities of the form $\mathbf{P}(\lambda_1Z_1+ \cdots +\lambda_kZ_k \geq w)$, proving the upper bound in Theorem \ref{thm:upper} and the lower bound in \eqref{eq:lowerness}.
\item In the final Section \ref{sec:lower} we complete our argument with a proof of Theorem \ref{thm:lower}.  
\end{itemize}
Throughout the article, $c,C > 0$ denote universal constants that may change from line to line. A lower case $c$ refers to a universal constant that is sufficiently small, and an upper case $C$ refers to a universal constant that is sufficiently large.

%%%%%%%%%%%%%%%%%%%%%%%%%%%%%%%%%%%%%%%%%%%%%%%%%%%%%%%%%%%%%
%%%%%%%%%%%%%%%%%%%%%%%%%%%%%%%%%%%%%%%%%%%%%%%%%%%%%%%%%%%%%
\section{Optimal transport ideas} \label{sec:OT}
%%%%%%%%%%%%%%%%%%%%%%%%%%%%%%%%%%%%%%%%%%%%%%%%%%%%%%%%%%%%%
%%%%%%%%%%%%%%%%%%%%%%%%%%%%%%%%%%%%%%%%%%%%%%%%%%%%%%%%%%%%%

%%%%%%%%%%%%%%%%%%%%%%%%%%%%%%%%%%%%%%%%%%%%%%
\subsection{Couplings and copulas}
%%%%%%%%%%%%%%%%%%%%%%%%%%%%%%%%%%%%%%%%%%%%%%
 Let $\mu_1,\ldots,\mu_k$ be probability measures on $\mathbb{R}$. A coupling $\Pi$ of $\mu_1,\ldots,\mu_k$ is a probability measure on $\mathbb{R}^k$ whose marginals are given by $\mu_1,\ldots,\mu_k$. In other words, $\Pi$ is a coupling of $\mu_1,\ldots,\mu_k$ if
\begin{align*}
\Pi( \{ x \in \mathbb{R}^k : x_i \in A \} ) = \mu_i(A)
\end{align*}
for each $1 \leq i \leq k$ and every Borel subset $A$ of $\mathbb{R}$. Writing $Z = (Z_1,\ldots,Z_k)$ for the coordinates of a random variable distributed according to $\Pi$, we will also refer to $\Pi$ as a coupling of the random variables $Z_1,\ldots,Z_k$.

The quantile function of a probability measure $\mu$ is the unique non-decreasing right-continuous function $Q:(0,1)\to\mathbb{R}$ satisfying
\begin{align} \label{eq:qrel}
\int_{-\infty}^{\infty} f(x)\,\mu(\mathrm{d}x)
=
\int_0^1 f(Q(r))\,\mathrm{d}r
\end{align}
for every bounded measurable function $f:\mathbb{R}\to\mathbb{R}$.

If $Q$ is the quantile function of a measure of the form $\mu = \frac{1}{n}\sum_{j=1}^n \delta_{a_j}$, then $Q$ is constant on each interval of the form $[(j-1)/n,j/n)$.

We write $\Phi:\mathbb{R} \to [0,1]$ for the distribution function of the standard Gaussian density, and $\Phi^{-1}:(0,1) \to \mathbb{R}$ for the associated quantile function. 

A $k$-dimensional copula $\pi$ is a coupling of $(\mu_1,\ldots,\mu_k)$ in the case where
\[\mu_1 = \cdots = \mu_k = \text{Lebesgue measure on $[0,1]$}.\] 
Given a copula $\pi$ and probability measures $\mu_1,\ldots,\mu_k$ with cumulative distribution functions $F_1,\ldots,F_k$, we may associate a coupling of $\mu_1,\ldots,\mu_k$ by letting
\begin{align} \label{eq:CC}
\Pi ( \{ x \in \mathbb{R}^k : x_1 \leq y_1,\ldots,x_k \leq y_k \} ) := \pi( \{ r \in [0,1]^k : r_1 \leq F_1(y_1),\ldots,r_k\leq F_k(y_k) \}).
\end{align} 
This sets up a correspondence between couplings of $\mu_1,\ldots,\mu_k$ and copulas, which is one-to-one whenever $\mu_1,\ldots,\mu_k$ have no atoms.
If the $\mu_i$ have atoms, then distinct copulas may give rise to the same coupling, see e.g.\ \cite[Section 3.1]{AJ}. Another way of looking at the relation \eqref{eq:CC} is that
\begin{align} \label{eq:CC2}
(U_1,\ldots,U_k) \sim \pi \implies (Q_1(U_1),\ldots,Q_k(U_k)) \sim \Pi,
\end{align}
where $Q_i$ is the quantile function of $\mu_i$.
Thus every coupling $\Pi$ of probability laws $\mu_1,\ldots,\mu_k$ arises through a $k$-tuple of marginally uniform random variables $(U_1,\ldots,U_k)$ distributed according to some copula $\pi$. We will make good use of this fact in both the remainder of the present section and also in Section \ref{sec:lower}.

\subsection{Empirical coordinate measures under vector addition}
In this section we explore how empirical coordinate measures behave under taking convex combinations of vectors. 

To set up this idea, we use a construction from \cite{CHS}. Given a $k$-tuple $\bm \sigma = (\sigma_1,\ldots,\sigma_k)$ of permutations in $\mathcal{S}_n$ we define their coupling measure $C^{\bm \sigma}(\mathrm{d}r) = C^{\bm \sigma}(r)\mathrm{d}r$ to be the $k$-dimensional copula with probability density function $C^{\bm \sigma}:[0,1]^k \to [0,\infty)$ given by 
\begin{align} \label{eq:cd}
C^{\bm \sigma}(r) := n^{k-1}\sum_{j=1}^n\prod_{i=1}^k \mathrm{1} \left\{ r_i \in \left[ \frac{\sigma^{-1}_i(j)-1}{n} , \frac{\sigma^{-1}_i(j)}{n} \right) \right\}.
\end{align} 

As a slight abuse of notation, we will interchange between referring to  $C^{\bm \sigma}$ as a density function on $[0,1]^k$ and as a probability measure on $[0,1]^k$.

The following lemma captures how empirical coordinate measures of convex combinations of vectors are encoded by coupling measures $C^{\bm \sigma}$:

\begin{lem}  \label{lem:probrep}
Let $s_1,\ldots,s_k$ be vectors in $\mathbb{R}^n$, let $\mu_{1},\ldots,\mu_{k}$ be the laws of their empirical coordinate measures and let $\sigma_{1},\ldots,\sigma_{k}$ be their coordinate ordering permutations (i.e.\ $\mu_i := \mu_{s_i}$ and $\sigma_i := \sigma_{s_i}$). Let $Q_1,\ldots,Q_k$ be the respective quantile functions of $\mu_1,\ldots,\mu_k$. Then the empirical coordinate measure $\mu_s$ of the convex combination $s = \lambda_1 s_1 + \cdots + \lambda_k s_k$ is the law of the random variable
\begin{align*}
Z = \lambda_1 Q_1(U_1) + \cdots + \lambda_k Q_k(U_k),
\end{align*}
where $(U_1,\ldots,U_k)$ is distributed according to the coupling measure $C^{\bm \sigma}$ on $[0,1]^k$ associated with the $k$-tuple $\bm \sigma = (\sigma_1,\ldots,\sigma_k)$. 
\end{lem}
\begin{proof}
For $1 \leq j \leq n$, consider the hypercube
\begin{align*}
H^{\bm \sigma}_j := \left\{ r \in [0,1]^k : \text{For each $1 \leq i \leq k$, } r_i \in \left[ \frac{\sigma^{-1}_i(j)-1}{n} , \frac{\sigma^{-1}_i(j)}{n} \right) \right\}.
\end{align*}
We have $C^{\bm \sigma}(H^{\bm \sigma}_j) = 1/n$ for each $j=1,\ldots,n$. Moreover, we note that if $r_i \in [(j-1)/n,j/n)$, then $Q_i(r_i)$ is equal to the $j^{\text{th}}$-smallest coordinate of $s_i$, which is $s_i^{\sigma_i(j)}$. Consequently, if $r_i \in [(\sigma_i^{-1}(j)-1)/n, \sigma_i^{-1}(j)/n)$, then $Q_i(r_i) = s_i^j$. 
It follows that on the event $\{ (U_1,\ldots,U_k) \in H^{\bm \sigma}_j\}$ we have $\lambda_1 Q_1(U_1)+ \cdots + \lambda_k Q_k(U_k) = \lambda_1 s_1^j + \cdots + \lambda_k s_k^j = s^j$, 
where $s^j$ is the $j^{\text{th}}$ coordinate of $s = \lambda_1 s_1 + \cdots + \lambda_k s_k$. 

Provided that the coordinates of $s$ are distinct, it follows that $\mathbf{P}( Z = s^j ) =1/n$. If $s$ has $\ell \geq 2$  coordinates equal to $s^j$, then $\mathbf{P}( Z= s^j) = \ell/n$. In any case, it follows that $Z$ is distributed according to the empirical coordinate measure of $s$.

\end{proof}

In particular, the empirical coordinate measure $\mu_{\lambda_1 s_1 + \cdots + \lambda_k s_k}$ is the distribution of a random variable $\lambda_1 Z_1 + \cdots + \lambda_k Z_k$ under some coupling of the random variables $Z_1,\ldots,Z_k$ with marginals $Z_i \sim \mu_{s_i}$.

%%%%%%%%%%%%%%%%%%%%%%%%%%%%%%%%%%%%%%%%%%%%%%%%%%
\subsection{Wasserstein distance and basic closure properties} \label{sec:Wasser}
%%%%%%%%%%%%%%%%%%%%%%%%%%%%%%%%%%%%%%%%%%%%%%%%%%

The Wasserstein distance between two probability measures $\mu_1$ and $\mu_2$ on $\mathbb{R}$ is defined by 
\begin{align} \label{eq:wasserdef}
W(\mu_1,\mu_2 ) := \inf_{ \Pi } \int_{\mathbb{R}^2}  |x_2-x_1|~ \Pi(\mathrm{d}x),
\end{align}
where the infimum is taken over all couplings $\Pi$ of $\mu_1$ and $\mu_2$. It is possible to show that if $\mu_1$ and $\mu_2$ have respective distribution functions $F_1,F_2: \mathbb{R} \to [0,1]$, or alternatively quantile functions $Q_1,Q_2 :(0,1) \to \mathbb{R}$, then we have
\begin{align} \label{eq:wasser}
W(\mu_1,\mu_2) := \int_{-\infty}^\infty |F_2(x) - F_1(x)| \mathrm{d}x = \int_0^1 |Q_2(r) - Q_1(r)| \mathrm{d}r;
\end{align}
see \cite[Proposition 2.17]{santa}. 

There is a natural coupling that achieves the infimum in \eqref{eq:wasserdef}. Namely, let $U$ be a random variable uniformly distributed on $[0,1]$, and let $\Pi$ be the law of the pair $(Q_1(U),Q_2(U))$. Then $Q_1(U)$ and $Q_2(U)$ respectively have the laws $\mu_1$ and $\mu_2$, and 
\begin{align} \label{eq:uniwasser}
W(\mu_1,\mu_2) = \mathbf{E}[|Q_2(U)-Q_1(U)|].
\end{align}

A ball of radius $\delta$ in the Wasserstein metric around a probability measure $\mu$ on $\mathbb{R}$ is simply a set of the form
\begin{align*}
B( \mu, \delta ) := \{ \nu \text{ probability measure on $\mathbb{R}$}: W(\nu,\mu) \leq \delta \}.
\end{align*}
More generally, if $S$ is a set of probability measures on $\mathbb{R}$, we write
\begin{align*}
B( S,\delta) := \bigcup_{ \mu \in S} B(\mu,\delta).
\end{align*}
Recall that if $\gamma$ is the standard Gaussian law, we define $E_n(\delta) := \{ s \in \mathbb{R}^n : W(\mu_s,\gamma) \leq \delta \} = \{ s \in \mathbb{R}^n : \mu_s \in B(\gamma,\delta) \}$.

We now study how Wasserstein distances interact with convex combinations of random variables under couplings.

% We turn to some basic results describing the interaction between Wasserstein distances and couplings.
Recall that $\mathcal{M}_k$, defined in \eqref{eq:Mkdef}, is the set of probability laws that govern the sum of a convex combination of $k$ Gaussian random variables. More generally, let us define
\begin{align*}
&\mathcal{M}(\mu_1,\ldots,\mu_k)\\
&:= \{ \mu: \text{ $\mu$ is the law of convex combination $\lambda_1 Z_1 + \cdots + \lambda_k Z_k$, $Z_i \sim \mu_i$} \} 
\end{align*}
to be the set of laws that occur as a convex combination of $k$ random variables with marginal laws $\mu_1,\ldots,\mu_k$ under any coupling. Of course we have
$\mathcal{M}_k := \mathcal{M}(\gamma,\ldots,\gamma)$ for the case when $\mu_1 = \cdots = \mu_k = \gamma$. 

\begin{lem} \label{lem:stable}
Let $\lambda_1,\ldots,\lambda_k \in \mathbb{R}$. 
Let $\mu_1,\ldots,\mu_k$ and $\tilde{\mu}_1,\ldots,\tilde{\mu}_k$ be probability measures on $\mathbb{R}$ with associated quantile functions $Q_1,\ldots,Q_k$ and $\tilde{Q}_1,\ldots,\tilde{Q}_k$. Suppose we have probability laws $\mu$ and $\tilde{\mu}$ defined as laws of random variables 
\begin{align*} 
\lambda_1 Q_1(U_1) + \cdots + \lambda_k Q_k(U_k) \sim \mu \qquad \text{and} \qquad \lambda_1 \tilde{Q}_1(U_1) + \cdots + \lambda_k \tilde{Q}_k(U_k) \sim \tilde\mu,
\end{align*}
where $(U_1,\ldots,U_k)$ are distributed according to some copula $\pi$ on $[0,1]^k$.

Then
\begin{align*}
W(\mu,\tilde{\mu}) \leq \sum_{i=1}^k |\lambda_i| W(\mu_i,\tilde{\mu}_i).
\end{align*}
\end{lem}
\begin{proof}
Using the definition of Wasserstein distance to obtain the first inequality below, and then the triangle inequality to obtain the second, we have 
\begin{align*}
W(\mu,\tilde{\mu}) &\leq \mathbf{E} \left[\left|(\lambda_1 Q_1(U_1) + \cdots + \lambda_k Q_k(U_k)) - (\lambda_1 \tilde{Q}_1(U_1) + \cdots + \lambda_k \tilde{Q}_k(U_k) ) \right| \right]\\
&\leq \sum_{i=1}^k| \lambda_i| \mathbf{E}[|Q_i(U_i) - \tilde{Q}_i(U_i)|]= \sum_{i=1}^k |\lambda_i |W(\mu_i,\tilde{\mu}_i),
\end{align*}
where the final equality above follows from \eqref{eq:uniwasser}. That completes the proof. 
\end{proof}
We highlight the following corollary:
\begin{cor} \label{cor:stable}
If $\mu_1,\ldots,\mu_k$ are elements of $B(\gamma,\delta)$, then
\begin{align*}
\mathcal{M}(\mu_1,\ldots,\mu_k) \subseteq B( \mathcal{M}_k,\delta).
\end{align*}
\end{cor}
\begin{proof}
Let $\mu \in \mathcal{M}(\mu_1,\ldots,\mu_k)$. Then $\mu$ is the law of a random variable of the form $\lambda_1 Q_1(U_1) + \cdots + \lambda_k Q_k(U_k)$ where $(U_1,\ldots,U_k)$ are distributed according to some copula $\pi$ on $[0,1]^k$, 
and $\lambda_i \in [0,1]$ with $\sum_{i=1}^k \lambda_i =1$. 

We may now construct an element $\tilde{\mu}$ of $\mathcal{M}_k$ by letting $\tilde{\mu}$ be the law of 
$\lambda_1 \Phi^{-1}(U_1) + \cdots + \lambda_k \Phi^{-1}(U_k)$. 

By the previous lemma, since $W(\mu_i,\gamma) \leq \delta$ for each $i$, it follows that $W(\mu,\tilde{\mu}) \leq \delta$. In particular, $\mu$ lies within a Wasserstein distance $\delta$ of an element $\tilde{\mu}$ in $\mathcal{M}_k$, as required. 

\end{proof}
We are now equipped to prove Proposition \ref{prop:comma}, which states that $\mathrm{conv}_k(E_n(\delta)) \subseteq \{ s \in \mathbb{R}^n : \mu_s \in B(\mathcal{M}_k,\delta) \}$.

\begin{proof}[Proof of Proposition \ref{prop:comma}]
Let $s \in \mathrm{conv}_k(E_n(\delta))$. Then $s$ can be written as a convex combination $s = \lambda_1 s_1 + \cdots + \lambda_k s_k$ for some $s_1,\ldots,s_k \in E_n(\delta)$. By Lemma \ref{lem:probrep}, $\mu_s$ lies in $\mathcal{M}(\mu_{s_1},\ldots,\mu_{s_k})$. Recalling that $E_n(\delta) := \{s \in \mathbb{R}^n : \mu_s \in B(\gamma,\delta)\}$, it follows from Corollary \ref{cor:stable} that $\mu_s \in B(\mathcal{M}_k,\delta)$.
\end{proof}

\subsection{Exceedances and Wasserstein distances}
Generalising an earlier definition, we define the exceedance at $w \in \mathbb{R}$ of a probability measure $\mu$ to be the amount of measure it gives to $[w,\infty)$, that is
\begin{align*}
\mathrm{Exc}_w(\mu) := \mu([w,\infty)).
\end{align*}
The exceedance takes values in $[0,1]$ and is a nonincreasing function of $w$.

If two measures are close in terms of Wasserstein distances, their exceedances are closely related:

\begin{lem} \label{lem:wasser}
If $W(\mu,\nu) \leq \delta$ then for any $w \in \mathbb{R}$ we have
\begin{align} \label{eq:law1}
\mathrm{Exc}_w(\mu) \geq \mathrm{Exc}_{w+\sqrt{\delta}}(\nu)- \sqrt{\delta}.
\end{align}
% and
% \begin{align} \label{eq:law2}
% \mathrm{Exc}_w(\mu) \leq \mathrm{Exc}_{w-\sqrt{\delta}}(\nu) + \sqrt{\delta}.
% \end{align}
\end{lem}
\begin{proof}
Let $F$ and $G$ denote the respective distribution functions of $\mu$ and $\nu$. Then $\mathrm{Exc}_w(\mu) := 1 - F(w-)$ and $\mathrm{Exc}_w(\nu) = 1 - G(w-)$. 
Here, $F(w-) := \lim_{ u \uparrow w} F(u)$ denotes the left limit of $F$ at $w$. In particular, using 
the first equality in \eqref{eq:wasser} to obtain the first equality below, we have
\begin{align*}
\delta \geq W(\mu,\nu) &= \int_{-\infty}^\infty|G(w) - F(w)| \mathrm{d}w = \int_{-\infty}^\infty |G(w-) - F(w-)| \mathrm{d}w\\
&= \int_{-\infty}^\infty |\mathrm{Exc}_w(\nu) - \mathrm{Exc}_w(\mu)| \mathrm{d}w \geq  \int_w^{w+\sqrt{\delta}}  (\mathrm{Exc}_u(\nu) - \mathrm{Exc}_u(\mu)) \mathrm{d}u\\
&\geq \sqrt{\delta} ( \mathrm{Exc}_{w+\sqrt{\delta}}(\nu) - \mathrm{Exc}_w(\mu)),
\end{align*}
where the final equality above follows from the fact that $\mathrm{Exc}_w(\mu)$ and $\mathrm{Exc}_w(\nu)$ are nonincreasing in $w$. Rearranging, we obtain \eqref{eq:law1}.
\end{proof}

%%%%%%%%%%%%%%%%%%%%%%%%%%%%%%%%%%%%%%%%%%%%%%%%%%%%%%%%%%%%%
%%%%%%%%%%%%%%%%%%%%%%%%%%%%%%%%%%%%%%%%%%%%%%%%%%%%%%%%%%%%%
\section{Proof of Proposition \ref{prop:DKW}} \label{sec:DKW}
%%%%%%%%%%%%%%%%%%%%%%%%%%%%%%%%%%%%%%%%%%%%%%%%%%%%%%%%%%%%%
%%%%%%%%%%%%%%%%%%%%%%%%%%%%%%%%%%%%%%%%%%%%%%%%%%%%%%%%%%%%%

In this brief section we prove Proposition \ref{prop:DKW}, which states that for large $n$, the set of points in $\mathbb{R}^n$ whose empirical coordinate measures lie within Wasserstein distance $\delta_n = C n^{-1/3}$ of the standard Gaussian distribution is overwhelmingly large in terms of $n$-dimensional Gaussian measure.

Our proof in this section relies on the Dvoretzky–Kiefer–Wolfowitz inequality \cite{DKW}. Recall that $\mu_s$ is the empirical coordinate measure of a vector $s$. Let
\begin{align*}
F_s(x) := \frac{1}{n} \# \{ 1 \leq j \leq n : s_j \leq x \}
\end{align*}
be the distribution function associated with $\mu_s$. 

The precise form of the Dvoretzky–Kiefer–Wolfowitz inequality we use states that for any $\lambda > 0$ we have

\begin{align} \label{eq:fed}
\gamma_n \left( \left\{ s \in \mathbb{R}^n : \sup_{ x \in \mathbb{R} } |F_s(x) - \Phi(x) | > \lambda/\sqrt{n} \right\} \right) \leq 2 e^{ - 2\lambda^2};
\end{align}
see Corollary 1 of \cite{massart}.

\begin{proof}[Proof of Proposition \ref{prop:DKW}]
By the union bound and a standard Gaussian tail bound we have 
\begin{align} \label{eq:fed1}
\gamma_n \left( \left\{ s \in \mathbb{R}^n : |s_j| \leq 2 \sqrt{\log(n)} ~ \text{for each $1 \leq j \leq n$} \right\} \right) 
% \geq 1 - 2n \int_{2 \sqrt{\log(n)}}^\infty \gamma(u)\mathrm{d}u 
\geq 1 - C/n.
\end{align}
Let 
\begin{align} \label{eq:fed2}
\Gamma_n := \left\{ s \in \mathbb{R}^n :  \sup_{ x \in \mathbb{R} } |F_s(x) - \Phi(x) | \leq \sqrt{\log(n)/n}, |s_j| < 2 \sqrt{\log(n)} ~\forall 1 \leq j \leq n \right\}.
\end{align}
Setting $\lambda = \sqrt{\log(n)}$ in \eqref{eq:fed} and combining the resulting bound with \eqref{eq:fed1} we have
\begin{align*}
\gamma_n( \Gamma_n) \geq 1 - C/n,
\end{align*}
for a possibly different universal constant $C > 0$.

We now show that every element $s$ of $\Gamma_n$ satisfies $W(\mu_s,\gamma) \leq C n^{-1/3}$. Indeed, using \eqref{eq:wasser} to obtain the initial equality below we have
\begin{align*}
&W(\mu_s, \gamma) = \int_{-\infty}^\infty | F_s(x) - \Phi(x) | \mathrm{d}x \\
&=  \int_{-2\sqrt{\log(n)}}^{2\sqrt{\log(n)}}  | F_s(x) - \Phi(x) | \mathrm{d}x\\
&+ \int_{-\infty}^{-2 \sqrt{\log(n)} }   | F_s(x) - \Phi(x) | \mathrm{d}x+  \int_{2 \sqrt{\log(n)} }^\infty    | F_s(x) - \Phi(x) | \mathrm{d}x.
\end{align*}
If $s \in \Gamma_n$, then $F_s(x) = 0$ for all $x \leq - 2 \sqrt{\log(n)}$, $F_s(x) = 1$ for all $x \geq 2 \sqrt{\log(n)}$, and $|F_s(x) - \Phi(x)| \leq \sqrt{\log(n)/n}$ for all $|x| \leq 2 \sqrt{\log(n)}$. Consequently, for $s \in \Gamma_n$ we have 
\begin{align*}
W(\mu_s, \gamma) & \leq 4\sqrt{\log(n)} \cdot \frac{\sqrt{\log(n)}}{\sqrt{n}} +  \int_{-\infty}^{-2 \sqrt{\log(n)} } \Phi(x) \mathrm{d}x +  \int_{2 \sqrt{\log(n)} }^\infty (1 - \Phi(x)) \mathrm{d}x \\
& \leq C \log(n)/\sqrt{n} \leq C n^{-1/3} =: \delta_n,
\end{align*}
where to obtain the second inequality above we have used the fact that $\int_{-\infty}^{-r} \Phi(x)\mathrm{d}x \leq C e^{ - cr^2 }$ with $r = \sqrt{2 \log(n)}$. 

That proves that every element $s$ of $\Gamma_n$ has $W(\mu_s,\gamma) \leq Cn^{-1/3}$, which implies $\Gamma_n \subseteq E_n(\delta_n)$. Since $\Gamma_n$ has Gaussian measure at least $1 - C/n$, so does $A_n := E_n(\delta_n)$, completing the proof of Proposition \ref{prop:DKW}.

 \end{proof}

%%%%%%%%%%%%%%%%%%%%%%%%%%%%%%%%%%%%%%%%%%%%%%%%%%%%%%%%%%%%%
%%%%%%%%%%%%%%%%%%%%%%%%%%%%%%%%%%%%%%%%%%%%%%%%%%%%%%%%%%%%%
\section{Upper and lower bounds for exceedances of convex sums of Gaussian random variables} \label{sec:upper}
%%%%%%%%%%%%%%%%%%%%%%%%%%%%%%%%%%%%%%%%%%%%%%%%%%%%%%%%%%%%%
%%%%%%%%%%%%%%%%%%%%%%%%%%%%%%%%%%%%%%%%%%%%%%%%%%%%%%%%%%%%%

This section is dedicated to proving upper and lower bounds for the supremum $S_k$ of exceedance probabilities $\mathbf{P}_\Pi(\lambda_1 Z_1 + \cdots + \lambda_k Z_k \geq 1)$, taken over all couplings $\Pi$ and all convex combinations.

\subsection{Preliminaries} \label{sec:prelim}

Let $\gamma(u) = (2\pi)^{-1/2}e^{-u^2/2}$ denote the standard one-dimensional Gaussian probability density function and let
\begin{align*}
\Phi(x) := \int_{-\infty}^x \gamma(u) \mathrm{d}u
\end{align*}
be its cumulative distribution function.

Using the fact $\gamma'(u) = - u\gamma(u)$ and $\gamma(u) = \gamma(-u)$, we have the relations
\begin{align} \label{eq:rel1}
\Phi'(x) = \gamma(x) = - \int_{-\infty}^x u \gamma(u)\mathrm{d}u = \int_{-x}^\infty u \gamma(u) \mathrm{d}u= \int_x^\infty u \gamma(u) \mathrm{d}u,
\end{align}

and
\begin{align} \label{eq:rel2}
 \Phi''(x) = -x\Phi'(x).
\end{align}
Note that if $Z$ is standard Gaussian, then
\begin{align} \label{eq:condexp}
\mathbf{E}[ Z | Z > - y ] = \frac{\int_{-y}^\infty u \gamma(u)\mathrm{d}u }{ \int_{-y}^\infty \gamma(u)\mathrm{d}u } =  \frac{ \Phi'(y)}{ \Phi(y)}.
\end{align}
Observe that $R(y) := \Phi'(y)/\Phi(y)$ is a monotone decreasing function in $y$, and satisfies $\lim_{y \downarrow -\infty} R(y) = + \infty$ and $\lim_{y \uparrow +\infty} R(y) = 0$. It follows that for each $w > 0$ there is a unique $y_w \in \mathbb{R}$ and a unique $p_w \in (0,1)$ satisfying 
\begin{align} \label{eq:py}
\Phi'(y_w)/\Phi(y_w) = w \qquad \text{and} \qquad p_w = \Phi(y_w).
\end{align}
From \eqref{eq:condexp} and \eqref{eq:py}, $-y_w < w$, so that $y_w + w > 0$. 

Alternatively, if we define for $p \in (0,1]$ the function
\begin{align} \label{eq:Qdef}
Q(p) := \Phi'(\Phi^{-1}(p))/p,
\end{align}
then $Q(p)$ is also monotone decreasing with $\lim_{p \downarrow 0} Q(p) = + \infty$ and $Q(1) = 0$, and $p_w$ is the unique solution to $Q(p_w) = w$.

%%%%%%%%%%%%%%%%%%%%%%%%%%%%%%%%%%%%%%%%%%%%%%%%%%%%%%%%%%%%%
%%%%%%%%%%%%%%%%%%%%%%%%%%%%%%%%%%%%%%%%%%%%%%%%%%%%%%%%%%%%%
\subsection{Upper bounds} \label{sec:upper2}
%%%%%%%%%%%%%%%%%%%%%%%%%%%%%%%%%%%%%%%%%%%%%%%%%%%%%%%%%%%%%
%%%%%%%%%%%%%%%%%%%%%%%%%%%%%%%%%%%%%%%%%%%%%%%%%%%%%%%%%%%%%

In this section, we prove Theorem \ref{thm:upper}, which states that if $s \in \mathrm{conv}_k(E_n(\delta))$, then its exceedance satisfies $\mathrm{Exc}_1(\mu_s) \leq p_1 - c e^{ - Ck^2 } + C \sqrt{\delta}$ for some universal constants $c,C > 0$. 

In fact, the bulk of our work will be in proving the preliminary optimal transport bound, Theorem \ref{thm:MK}, which is an upper bound on 
\begin{align*}
S_k := \sup_\Pi \sup_{\lambda_1,\ldots,\lambda_k} \mathbf{P}_\Pi ( \lambda_1 Z_1 + \cdots + \lambda_k Z_k \geq 1 ).
\end{align*}

In fact, we will work more generally by controlling suprema of $\mathbf{P}_\Pi \left(  \lambda_1 Z_1 + \cdots + \lambda_k Z_k \geq w \right)$ for $w > 0$. Before using more sophisticated methods from optimal transport to bound $S_k$, we will begin by using a simple first moment argument to prove an upper bound on $S_k$ that is universal in all couplings, all integers $k\geq 1$, and all convex combinations. While we will not have recourse to use this first moment argument directly at any point in the sequel, its proof involves a simple calculation that will help motivate the sharper results we obtain later.

\begin{lem} \label{lem:firstmoment}
Let $w > 0$. Then for all couplings $\Pi$ of a $k$-tuple $(Z_1,\ldots,Z_k)$ of standard Gaussian random variables, and all $\lambda_i \in [0,1]$ with $\sum_{i=1}^k \lambda_i = 1$, we have 
\begin{align} \label{eq:firstmoment}
\mathbf{P}_\Pi \left( \lambda_1 Z_1 + \ldots + \lambda_k Z_k \geq w \right) \leq p_w,
\end{align}
where $p_w$ is defined in \eqref{eq:py}. 
\end{lem}
\begin{proof}
Let $\Gamma := \{\lambda_1 Z_1 + \cdots + \lambda_k Z_k \geq w \}$ and let $p = \mathbf{P}_\Pi(\Gamma)$. 
 
By linearity and the fact that $\mathbf{E}_\Pi[Z_i]= 0$ for each $1 \leq i \leq k$, we have
\begin{align} \label{eq:moon0}
0 = \mathbf{E}_\Pi [ ( \lambda_1 Z_1 + \cdots + \lambda_k Z_k ) \mathrm{1}_\Gamma ]  + \sum_{i=1}^k \lambda_i \mathbf{E}_\Pi [ Z_i \mathrm{1}_{\Gamma^c} ].
\end{align}
We now look for a lower bound on the right-hand side of \eqref{eq:moon0}. Note that if $B$ is any event with probability $p'$, and $Z$ is a standard Gaussian random variable, then
\begin{align} \label{eq:77}
\mathbf{E}[Z \mathrm{1}_{B} ] \geq \int_{-\infty}^{\Phi^{-1}(p')} u \gamma(u)\mathrm{d}u = - \Phi'(\Phi^{-1}(p')). 
\end{align}
Setting $B := \Gamma^c$ to be the complement of $\Gamma$ (so that $p' = 1-p$) and using the definition of $\Gamma$, we obtain from \eqref{eq:moon0} and \eqref{eq:77} the inequality 
\begin{align} \label{eq:moon}
0 \geq pw - \Phi'(\Phi^{-1}(1-p)).
\end{align}
By symmetry, $\Phi'(\Phi^{-1}(1-p)) = \Phi'(\Phi^{-1}(p))$, so that with the notation of \eqref{eq:Qdef}, after some rearrangement \eqref{eq:moon} reads $Q(p) \geq w$. Since $Q(p)$ is monotone decreasing, the inequality $Q(p) \geq w = Q(p_w)$ implies $p \leq p_w$, which is precisely the statement of the result.

\end{proof}

We now refine the first moment bound in Lemma \ref{lem:firstmoment} using the Monge-Kantorovich duality idea outlined in equation \eqref{eq:MKduality00}. 
Observe that if $f_1,\ldots,f_k:\mathbb{R} \to \mathbb{R}$ are any measurable functions satisfying 
\begin{align} \label{eq:MKpre}
\mathrm{1} \{ \lambda_1 z_1 + \ldots + \lambda_k z_k \geq w \} \leq \sum_{i=1}^k f_i(z_i),
\end{align}
then by taking expectations under $\Pi$ through \eqref{eq:MKpre} we obtain in the setting of Lemma \ref{lem:firstmoment} the inequality
\begin{align} \label{eq:MKupper}
\mathbf{P}_\Pi \left(  \lambda_1 Z_1 + \cdots + \lambda_k Z_k \geq w \right) \leq \mathbf{E}_\Pi \left[ \sum_{i=1}^k f_i(Z_i) \right] = \sum_{i=1}^k \mathbf{E}_\gamma[ f_i(Z) ],
\end{align}
where $Z$ is a one-dimensional standard Gaussian random variable under $\gamma$. Monge-Kantorovich duality in our case then states that in fact we have the relation
\begin{align} \label{eq:MKduality}
\mathrm{sup}_\Pi \mathbf{P}_\Pi \left(  \lambda_1 Z_1 + \cdots + \lambda_k Z_k \geq w \right) = \inf_{ f_1,\ldots,f_k }\sum_{i=1}^k \mathbf{E}_\gamma[ f_i(Z) ],
\end{align}
where the supremum is taken over all couplings $\Pi$ of $(Z_1,\ldots,Z_k)$, and the infimum is taken over all $f_1,\ldots,f_k$ satisfying \eqref{eq:MKpre}. 

To be clear, in our proof of the upcoming Theorem \ref{thm:MK2}, we will not invoke Monge-Kantorovich duality \eqref{eq:MKduality} at any stage. Rather, we simply aim to control the probability $\mathbf{P}_\Pi \left(  \lambda_1 Z_1 + \cdots + \lambda_k Z_k \geq w \right)$ using \eqref{eq:MKupper} in the knowledge that \eqref{eq:MKduality} confirms this has the potential to be a good strategy with a suitable choice of $f_1,\ldots,f_k$. 
 
We are now ready to prove Theorem \ref{thm:MK}. In fact, Theorem \ref{thm:MK} follows from setting $w=1$ in the following more general statement which we now state and prove:

\begin{thm} \label{thm:MK2}

Let $w \in [1/2,3/2]$. Then for all couplings $\Pi$ of a $k$-tuple $(Z_1,\ldots,Z_k)$ of standard Gaussian random variables, and all $\lambda_i \in [0,1]$ with $\sum_{i=1}^k \lambda_i = 1$, we have 
\begin{align} \label{eq:krefine}
\mathbf{P}_\Pi \left( \lambda_1 Z_1 + \ldots + \lambda_k Z_k \geq w \right) \leq p_w - c e^{ - Ck^2}.
\end{align}
Here $c,C > 0$ are universal constants that do not depend on $w \in [1/2,3/2]$, the integer $k \geq 1$, or on the reals $\lambda_1,\ldots,\lambda_k \in [0,1]$.
\end{thm}

Before proving Theorem \ref{thm:MK2}, we will give an alternative proof of Lemma \ref{lem:firstmoment} using the equation \eqref{eq:MKupper}. Our proof of Theorem \ref{thm:MK2} will then refine this approach.

The notation $x \vee y$ (resp.\ $x \wedge y$) denotes the maximum (resp.\ the minimum) of real numbers $x$ and $y$. 

\begin{proof}[Alternative proof of Lemma \ref{lem:firstmoment} using \eqref{eq:MKupper}]

Let $y > -w$ be a variable, and define $a_y := 1/(y+w)$, so that the function $x \mapsto 1 + a_y(x-w)$ equals zero when $x = -y$. It follows that if we let $g_y(x)$ be the nonnegative part of this function we can write 
\begin{align*}
g_y(x) := (1 + a_y(x-w)) \vee 0 = \mathrm{1}_{\{x \geq - y\}} (1 + a_y(x-w)).
\end{align*}
Setting $f_i(x) = \lambda_i g_y(x)$, we now verify that \eqref{eq:MKpre} is satisfied. Indeed, 
\begin{align*}
\sum_{i=1}^k f_i(z_i) &= \sum_{i=1}^k \lambda_i ( 1 + a_y(z_i-w)) \vee 0)\\
&\geq 0 \vee \sum_{i=1}^k \lambda_i( 1 + a_y(z_i - w))\\
&= 0 \vee \left( 1 + a_y \left( \sum_{i=1}^k \lambda_i z_i - w \right) \right)\\
& \geq \mathrm{1} \{ \lambda_1 z_1 + \cdots + \lambda_k z_k \geq w \},
\end{align*}
where in the final equality above we used the fact that $a_y > 0$ since $y > -w$. 

It follows that by \eqref{eq:MKupper} we have 
\begin{align} \label{eq:upp1}
\mathbf{P}_\Pi \left( \lambda_1 Z_1 + \ldots + \lambda_k Z_k \geq w \right) \leq \sum_{i=1}^k \mathbf{E}_\gamma[ f_i(Z)] = \mathbf{E}_\gamma[ g_y(Z)] =: H(y).
\end{align}
We now show that if $y_w$ is the minimiser of $H(y)$, then $H(y_w) = p_w$, so that from \eqref{eq:upp1} we achieve the bound in Lemma \ref{lem:firstmoment}. 
A brief calculation using \eqref{eq:rel1} and the definition of $a_y$ tells us that
\begin{align} \label{eq:Hy}
H(y) &= \int_{-y}^\infty (1 + a_y(x-w))\gamma(x) \mathrm{d}x \nonumber \\
&= (1-a_yw)\Phi(y) + a_y \Phi'(y) = \frac{1}{y+w}( y \Phi(y) + \Phi'(y) ).
\end{align}
Using \eqref{eq:Hy} to differentiate $\log H(y)$ to obtain the first equality below, and then using \eqref{eq:rel2} to obtain the second, we have
\begin{align*}
\frac{\mathrm{d}}{\mathrm{d}y} \log H(y) = - \frac{1}{y+w} + \frac{\Phi(y) + y \Phi'(y) + \Phi''(y)}{y\Phi(y)+\Phi'(y)} = - \frac{1}{y+w} + \frac{1}{y+\Phi'(y)/\Phi(y)}. 
\end{align*}
Thus $H(y)$ has a stationary point when $y$ satisfies $\Phi'(y)/\Phi(y) = w$, i.e.\ $y = y_w$ as in \eqref{eq:py}. Plugging $y = y_w$ into \eqref{eq:Hy} and using $\Phi'(y_w) = w\Phi(y_w)$ we obtain
\begin{align} \label{eq:upp2}
H(y_w ) = \Phi(y_w) = p_w,
\end{align}
where $p_w$ is as in its definition in \eqref{eq:py}. 
Substituting \eqref{eq:upp2} into \eqref{eq:upp1} yields the bound \eqref{eq:firstmoment}, completing the alternative proof of Lemma \ref{lem:firstmoment} using \eqref{eq:MKupper}.
\end{proof}

In the alternative proof of Lemma \ref{lem:firstmoment}, we set $f_i(x) := \lambda_i g_y(x)$ and then optimized over $y$. We now refine the function $f_i(x)$ occurring in this previous proof by better taking into account the quantities $\lambda_1,\ldots,\lambda_k$. This refinement leads to the $k$-dependent improvement in \eqref{eq:krefine}.

\begin{proof}[Proof of Theorem \ref{thm:MK2}]
Let us hereon write $f_i(x) = \lambda_i g_{y_w}(x)$ where $y_w$ is the minimiser of $H(y)$ in the previous proof, so that $H(y_w) = \Phi(y_w) = p_w$. Let $a_w = 1/(y_w+w)$. Clearly, the functions $f_i(x)$ are nonnegative, so that if any one of them exceeds $1$ we automatically have $\sum_{i=1}^k f_i(x_i) \geq \mathrm{1} \{ \lambda_1 x_1 + \ldots + \lambda_k x_k \geq w \}$. Thus we have some room for improvement if we consider 
\begin{align*}
\tilde{f}_i(x) := f_i(x) \wedge 1
\end{align*}
instead of $f_i(x)$ in our upper bound, since \eqref{eq:MKpre} still holds with $\tilde{f}_1,\ldots,\tilde{f}_k$ in place of $f_1,\ldots,f_k$. Again by \eqref{eq:MKupper} we have
\begin{align} \label{eq:ff} 
\mathbf{P}_\Pi \left(  \lambda_1 Z_1 + \cdots + \lambda_k Z_k \geq w \right) \leq \sum_{i=1}^k \mathbf{E}_\gamma[ \tilde{f}_i(Z) ],
\end{align}
so that since $\mathbf{E}_\gamma [ \tilde{f}_i(Z) ] < \mathbf{E}_\gamma[ f_i(Z)]$ we are set to sharpen the previous bound.

We now calculate $\mathbf{E}_\gamma [ \tilde{f}_i(Z) ]$. Letting $x_i$ be the solution to $\lambda_i (1 + a_w(x_i - w)) = 1$, i.e.\ $x_i = w + (1/\lambda_i-1)/a_w$, we can write
\begin{align} \label{eq:urform}
\tilde{f}_i(x) = f_i(x) - \lambda_i a_w ( x - x_i) \mathrm{1}_{\{ x \geq x_i \}}.
\end{align}
Applying \eqref{eq:ff}, \eqref{eq:urform}, \eqref{eq:upp1} with $y = y_w$, and \eqref{eq:upp2}, we obtain
\begin{align} \label{eq:ff2} 
\mathbf{P}_\Pi \left(  \lambda_1 Z_1 + \cdots + \lambda_k Z_k \geq w \right) \leq p_w - \sum_{i=1}^k \lambda_i a_w \int_{x_i}^\infty (x-x_i) \gamma(x) \mathrm{d}x.
\end{align}

There are constants $c, C>0$ such that $\int_{r}^\infty (x-r) \gamma(x) \mathrm{d}x \geq c e^{ -C r^2 }$ whenever $r \geq 1/2$. Moreover, we note that $x_i \geq 1/2$ whenever $w \geq 1/2$. Finally, note that there exists some $c > 0$ such that $a_w \geq c$ for all $w \in [1/2,3/2]$. In particular, using \eqref{eq:ff2}, we have
\begin{align} \label{eq:ff3} 
\mathbf{P}_\Pi \left(  \lambda_1 Z_1 + \cdots + \lambda_k Z_k \geq w \right) \leq p_w - c \sum_{i=1}^k \lambda_ie^{ - C x_i^2}.
\end{align}

Again since $a_w > c > 0$ for $w \in [1/2,3/2]$, we have $x_i \leq C + C/\lambda_i$ for some $C > 0$. Using this bound in \eqref{eq:ff3} we obtain 
\begin{align} \label{eq:ff4} 
\mathbf{P}_\Pi \left(  \lambda_1 Z_1 + \cdots + \lambda_k Z_k \geq w \right) \leq p_w - c \sum_{i=1}^k \lambda_ie^{ - C/\lambda_i^2},
\end{align}

for possibly different universal constants $c,C>0$ still not depending on $\lambda_1,\ldots,\lambda_k$, on $k$, or on $w \in [1/2,3/2]$. Set $f(\lambda_1,\ldots,\lambda_k) = \sum_{i=1}^k \lambda_ie^{ - C/\lambda_i^2}$. It is a straightforward calculation using Lagrange multipliers to verify that for any $C > 0$ we have 
\begin{align} \label{eq:LM}
\inf_{ \lambda_1,\ldots,\lambda_k } f(\lambda_1,\ldots,\lambda_k) = f(1/k,\ldots,1/k) = e^{ - Ck^2},
\end{align}
where the infimum in \eqref{eq:LM} is taken over all $\lambda_1,\ldots,\lambda_k \in [0,1]$ satisfying $\sum_{i=1}^k \lambda_i = 1$. 
Using \eqref{eq:LM} in \eqref{eq:ff4}, we obtain \eqref{eq:krefine}.
\end{proof}

%%%%%%%%%%%%%%%%%%%%%%%%%%%%%%%%%%%%%%%%%%%

Theorem \ref{thm:upper} now follows fairly quickly from Theorem \ref{thm:MK2} together with results developed in Section \ref{sec:OT}:

\begin{proof}[Proof of Theorem \ref{thm:upper}]
Let $s \in \mathrm{conv}_k(E_n(\delta))$. Then by Proposition \ref{prop:comma}, $\mu_s \in B(\mathcal{M}_k,\delta)$. It follows from Lemma \ref{lem:wasser} that
\begin{align*}
\mathrm{Exc}_1(\mu_s) \leq  \sqrt{\delta} + \sup_{ \nu \in \mathcal{M}_k } \mathrm{Exc}_{1-\sqrt{\delta}}(\nu). 
\end{align*}
Using Theorem \ref{thm:MK2}, provided $\sqrt{\delta} \leq 1/2$ it follows that
\begin{align*}
\mathrm{Exc}_1(\mu_s) \leq  \sqrt{\delta} + p_{1-\sqrt{\delta}} - c e^{ - Ck^2}. 
\end{align*}
Finally, note that the function $p_w =p(w)$ is decreasing and smooth in $w$. In particular, there is some constant $C_1 \geq 0$ such that $p_{1-\sqrt{\delta}} \leq p_1 + C_1 \sqrt{\delta}$ for all $0 \leq \delta \leq 1/2$. It follows that with $C = C_1 + 1$ we have
\begin{align*}
\mathrm{Exc}_1(\mu_s) \leq p_{1} + C \sqrt{\delta} - c e^{ - Ck^2},
\end{align*}
thereby completing the proof of Theorem \ref{thm:upper}.

\end{proof}

%%%%%%%%%%%%%%%%%%%%%%%%%%%%%%%%%%%%%%%%%%%%%%%%%%%%%
\subsection{Lower bounds for exceedances}
%%%%%%%%%%%%%%%%%%%%%%%%%%%%%%%%%%%%%%%%%%%%%%%%%%%%%
Recall from Section \ref{sec:upper} that for $w > 0$ we have $y_w \in \mathbb{R}$ and $p_w \in (0,1)$ defined by $\Phi'(y_w)/\Phi(y_w) = w$ and $p_w = \Phi(y_w)$. By symmetry, we also have $- y_w = \Phi^{-1}(1-p_w)$. In particular, $p_w$ is chosen so that
\begin{align*}
\mathbf{E}[ Z | Z > \Phi^{-1}(1-p_w) ] = w.
\end{align*}
In other words, the conditional expectation of a Gaussian random variable given that it lies in its upper $p_w$-quantile is $w$. 

Suppose now that under a probability measure $P_\rho$, the variables $(Z_1,\ldots,Z_d)$ are independent and each has the law of a standard Gaussian random variable conditioned to exceed $\Phi^{-1}(1-p_1+\rho)$. Write $E_\rho[Z_1]$ for the expectation of $Z_1$ under this probability measure. 
We now study the small-$\rho$ behaviour of $E_\rho[Z_1]$. Of course, then $E_0[Z_1] = \mathbf{E}[ Z | Z > \Phi^{-1}(1-p_1)] = 1$. More generally, using \eqref{eq:rel1} to obtain the second equality below we have 
\begin{align*}
E_\rho[Z_1] = \frac{ \int_{ \Phi^{-1}(1-p_1+\rho) }^\infty u \gamma(u)\mathrm{d}u}{ p_1 - \rho } = \frac{  \Phi'( \Phi^{-1}(1-p_1+\rho) )}{ p_1 - \rho}.
\end{align*}
We now seek to differentiate $E_\rho[Z]$ with respect to $\rho$. Set $f(\rho) := \Phi^{-1}(1 - p_1+\rho)$. Then by the inverse function theorem $f'(\rho) = 1/\Phi'(f(\rho))$. Using this fact to obtain the first equality below, and then $\Phi''(x) = - x \Phi'(x)$ (see \eqref{eq:rel2}) to obtain the second, we have 
\begin{align*}
\frac{\mathrm{d}}{\mathrm{d}\rho} E_\rho[Z_1] &= \frac{ \Phi''(f(\rho)) }{ (p_1 - \rho) \Phi'(f(\rho)) } + \frac{1}{(p_1-\rho)^2} \Phi'(f(\rho))\\
&= - \frac{ f(\rho)}{ p_1 - \rho } + \frac{1}{(p_1 - \rho)^2} \Phi'(f(\rho)).
\end{align*}
Using \eqref{eq:py} we have $f(0) = \Phi^{-1}(1-p_1) = -y_1$ and $\Phi'(-y_1) = \Phi'(y_1) = \Phi(y_1) = p_1$. It follows that 
\begin{align} \label{eq:rhomove}
\frac{\mathrm{d}}{\mathrm{d}\rho} E_\rho[Z_1]|_{\rho=0} = \frac{1+y_1}{p_1} =: \kappa > 0.
\end{align}
A heuristic explanation for \eqref{eq:rhomove} is that as $\rho$ increases a small amount, we are replacing mass near the lower tail, which occurs at $-y_1$, with mass in the rest of the distribution, which has average $1$. Thus, the rate of change in expectation per change in $\rho$ is $\kappa := (1 - (-y_1))/p_1$. It follows from \eqref{eq:rhomove} that for small $\rho$,
\begin{align*}
E_\rho[Z_1] = 1 + \kappa \rho + O(\rho^2).
\end{align*}

With this picture in mind, we have the following lemma.

\begin{lem} \label{lem:boxdiag}
Under a probability law $P_\rho$, let $Z_1,\ldots,Z_d$ be independent random variables that have the law of a standard Gaussian random variable conditioned to exceed $\Phi^{-1}(1 - p_1 + \rho)$. 

Then there exists $c > 0$ such that for all $0 < \rho < c$ we have 
\begin{align} \label{eq:cherny}
P_\rho \left( \frac{Z_1 + \cdots + Z_d}{d} \leq 1 + \frac{\kappa}{2} \rho \right) \leq e^{ - c \rho^2 d}.
\end{align}
\end{lem}

Since this lemma is a variant on standard tail bounds for i.i.d.\ random variables with exponential moments, we will sketch the proof here, and relegate its full proof to Appendix A. 

\begin{proof}[Sketch proof of Lemma \ref{lem:boxdiag}]
Let $\rho$ be small. By the central limit theorem, under $P_\rho$, for large $d$, the law of $d^{-1}(Z_1+\cdots + Z_d)$ concentrates around $1 + \kappa \rho$ with fluctuations of the order $( 1 +O(\rho))V /\sqrt{d}$ where $V$ is the variance of $Z_1$ under $P_0$. Since $Z_1$ admits exponential moments of all orders under $P_\rho$, it follows from a standard Chernoff-type argument that there is a constant $c> 0$ such that for all $0 < r < c$ we have
\begin{align*}
P_\rho \left( \frac{Z_1 + \cdots + Z_d}{d} - (1+\kappa \rho + O(\rho^2)) \leq - r \right) \leq e^{ - c d r^2}.
\end{align*}
Setting $r = \frac{\kappa}{2}\rho$ and rearranging, we obtain \eqref{eq:cherny} (with a possibly different $c>0$). 
\end{proof}
The rigorous proof of Lemma \ref{lem:boxdiag} is of course more delicate than the sketch presented above suggests, as the law $P_\rho$ of the random variables $Z_1,\ldots,Z_d$ itself also depends on $\rho$.

We close this section with a corollary that provides a lower bound for the maximal exceedance of a probability measure $\mu$ in $\mathcal{M}_d$. While we will not use this result explicitly in the remainder of the article, we record it here as it acts as a complement to Theorem \ref{thm:MK}. This corollary uses a box-product coupling of Gaussian random variables. We recall from the introduction that the box-product coupling at $q$ is the coupling $\Pi$ of $d$ Gaussian random variables with probability density function on $\mathbb{R}^d$ given by 
\begin{align} \label{eq:boxproduct}
\Pi(x) := \left( (1-p)^{-(d-1)} \mathrm{1}_{\{ x_i < q ~\forall i=1,\ldots,d \}} + p^{-(d-1)} \mathrm{1}_{\{ x_i \geq q ~\forall i=1,\ldots,d \}}   \right) \gamma_d(x),
\end{align}
where $1-p := \Phi(q)$ and $\gamma_d(x) = (2\pi)^{-d/2}e^{ - (x_1^2+\cdots+x_d^2)/2}$ is the standard $d$-dimensional Gaussian density.

\begin{cor} \label{cor:underbound}
There are constants $c, C > 0$ such that if we set $\rho_d := c \sqrt{\log(d)/d}$ and let $\Pi$ denote the box-product coupling at $\Phi^{-1}(1 - p_1 + \rho_d)$, then
\begin{align*}
\mathbf{P}_\Pi \left( \frac{Z_1 + \cdots + Z_d}{d} \geq 1 \right) \geq p_1 - C \frac{ \log (d) }{ \sqrt{d} } .
\end{align*} 
\end{cor}

\begin{proof}
Let $0 < \rho < c$, where $c$ is as in the statement of Lemma \ref{lem:boxdiag}. Let $\Pi_\rho$ be the box-product coupling at $\Phi^{-1}(1 - p_1 + \rho)$. We have $\mathbf{P}_{\Pi_\rho} ( Z_1 \geq \Phi^{-1}(1 - p_1 + \rho ) ) = p_1 -\rho$. Moreover, under $\mathbf{P}_{\Pi_\rho}$ and conditional on $\{ Z_1 \geq \Phi^{-1}(1 - p_1 + \rho ) \}$, $(Z_1,\ldots,Z_d)$ have the law of $P_\rho$. Using Lemma \ref{lem:boxdiag} to obtain the penultimate inequality below we have
\begin{align*}
\mathbf{P}_{\Pi_\rho} ( d^{-1}(Z_1 + \cdots + Z_d) \geq 1 ) &\geq (p_1-\rho_d) P_\rho \left(  d^{-1}(Z_1 + \cdots + Z_d) \geq 1 \right) \\
&\geq (p_1-\rho_d) P_\rho \left(  d^{-1}(Z_1 + \cdots + Z_d) \geq 1 + \frac{\kappa}{2} \rho \right) \\
& \geq ( p_1 - \rho)(1 - e^{ - c \rho^2 d } ) \geq p_1 - \rho - e^{ - c\rho^2 d} .
\end{align*} 
Now by choosing $\rho = \rho_d = c \sqrt{\frac{\log(d)}{d} }$ for a suitable constant $c>0$, we obtain
\begin{align*}
\mathbf{P}_{\Pi_{\rho_d}}( d^{-1}(Z_1 + \cdots + Z_d) \geq 1)  \geq p_1 - C \sqrt{\frac{\log(d)}{d}}
\end{align*}
for some constant $C$, as required.
\end{proof}

%%%%%%%%%%%%%%%%%%%%%%%%%%%%%%%%%%%%%%%%%%%%%%%%%%%%%%%%%%%%%
%%%%%%%%%%%%%%%%%%%%%%%%%%%%%%%%%%%%%%%%%%%%%%%%%%%%%%%%%%%%%
\section{Proof of Theorem \ref{thm:lower}} \label{sec:lower}
%%%%%%%%%%%%%%%%%%%%%%%%%%%%%%%%%%%%%%%%%%%%%%%%%%%%%%%%%%%%%
%%%%%%%%%%%%%%%%%%%%%%%%%%%%%%%%%%%%%%%%%%%%%%%%%%%%%%%%%%%%%

\subsection{Large sets of ordering permutations}

Let $\bm \sigma := (\sigma_1,\ldots,\sigma_d)$ be a $d$-tuple of permutations in $\mathcal{S}_n$. Recall from \eqref{eq:cd} that the coupling density on $[0,1]^d$ associated with $\sigma_1,\ldots,\sigma_d$ is the probability density function $C^{\bm \sigma}:[0,1]^d \to [0,\infty)$ given by 
 \begin{align} \label{eq:cd2}
C^{\bm \sigma}(r) := n^{d-1}\sum_{j=1}^n\prod_{i=1}^d\mathrm{1} \left\{ r_i \in \left[ \frac{\sigma^{-1}_i(j)-1}{n} , \frac{\sigma^{-1}_i(j)}{n} \right) \right\}.
\end{align} 

Recall from \eqref{eq:CC} that each coupling is associated with a copula. The copula associated with the box-product coupling at $q = \Phi^{-1}(1-p)$ (see \eqref{eq:boxproduct}) has probability density function $\pi:[0,1]^d \to [0,\infty)$ given by 
\begin{align} \label{eq:boxcop}
\pi_p(r) = (1-p)^{-(d-1)}\mathrm{1}\{ r \in [0,1-p)^d \} + p^{-(d-1)} \mathrm{1} \{ r \in [1-p,1]^d \}.
\end{align} 

Given a subset $B$ of $\{1,\ldots,n\}$, define 
\begin{align*}
\mathcal{S}_n(B) := \{ \sigma \in \mathcal{S}_n : \sigma(\{ n-m+1,\ldots,n\} ) = B \},
\end{align*}
where $m = \# B$ is the cardinality of $B$. 

For large $n$, we are interested in finding $d$-tuples of permutations $\bm \sigma = (\sigma_1,\ldots,\sigma_d)$ for which the coupling measure $C^{\bm \sigma}(r)\mathrm{d}r$ approximates $\pi_p(r)\mathrm{d}r$. As an initial step, we have the following lemma.

\begin{lem}
Let $p \in (0,1)$ such that $m = pn$ is an integer.
The measure $C^{\bm \sigma}$ is supported on the disjoint union $[0,1-p)^d \cup [1-p,1]^d$ if and only if there exists a subset $B \subseteq \{1,\ldots,n\}$ of cardinality $m$ such that $\sigma_i \in \mathcal{S}_n(B)$ for each $i = 1,\ldots,d$.
\end{lem}

\begin{proof}
By examining \eqref{eq:cd2} we see that $C^{\bm \sigma}$ is supported on $[0,1-p)^d \cup [1-p,1]^d$ if and only if for each $1 \leq j \leq n$ we have either 
\begin{align*}
\sigma_i^{-1}(j) \leq n-m \quad \forall ~i=1,\ldots,d \qquad \text{or} \qquad \sigma_i^{-1}(j) > n-m \quad  \forall ~i=1,\ldots,d .
\end{align*}
Equivalently, there exists a subset $B$  of cardinality $m$ such that each $\sigma_i$ maps $\{n-m+1,\ldots,n\}$ to $B$, that is, each $\sigma_i \in \mathcal{S}_n(B)$ for every $i$. 
\end{proof}

%%%%%%%%%%%%%%%%%%%%%%%%%%%%%%%%%%%%%%%%%%%%%%%%%%%%%%%%%%%%%%%%%
\subsection{Large reordering sets}
%%%%%%%%%%%%%%%%%%%%
Recall that if $s = (s^1,\ldots,s^n) \in \mathbb{R}^n$ is a vector and $\sigma \in \mathcal{S}_n$ is a permutation, we write $\sigma s = (s^{\sigma(1)},\ldots,s^{\sigma(n)})$. Note that reordering the coordinates does not change the empirical coordinate measure, i.e.\ $\mu_{\sigma s} = \mu_s$. In particular, $s \in E_n(\delta) \iff \sigma s \in E_n(\delta)$. 

Recall that $W_n := \{ t = (t^1,\ldots,t^n) \in \mathbb{R}^n : t^1 \leq \cdots \leq t^n \}$ is the subset of $\mathbb{R}^n$ consisting of vectors with nondecreasing coordinates. Given a subset $J$ of $\mathbb{R}^n$ and an element $t \in W_n$, we are interested in the occurrences of some reordering of the element $t$ in $J$. Define
\begin{align*}
N(t,J) := \# \{ \sigma \in \mathcal{S}_n : \sigma t \in J \} \quad \text{and} \quad N_B(t,J) := \# \{ \sigma \in \mathcal{S}_n(B) : \sigma t \in J \}. 
\end{align*}
Note that for each $\sigma \in \mathcal{S}_n$ and each $0 \leq m \leq n$, there exists a unique set $B$ of cardinality $m$ such that $\sigma \in \mathcal{S}_n(B)$. It follows that for each $0 \leq m \leq n$ we have the identity
\begin{align} \label{eq:Bsum}
\sum_{ \# B = m } N_B(t,J ) = N(t,J),
\end{align}
where the sum is taken over all subsets $B$ of $\{1,\ldots,n\}$ of cardinality $m$.

Observe that by symmetry we have $\gamma_n(W_n) = 1/n!$. (Equivalently, the probability that a random standard Gaussian vector in $n$ dimensions has its coordinates listed in nondecreasing order is $1/n!$.) Moreover,
\begin{align} \label{eq:xx}
\int_{W_n} N(t,J) \gamma_n(\mathrm{d}t) = \gamma_n(J).
\end{align}

Our next result tells us that for any reasonably large subset of $\mathbb{R}^n$, for each $m$ there exists some reasonably large $N_B(t,J)$ with $\# B = m$.

\begin{lem} \label{lem:existence}
Let $J$ be a measurable subset of $\mathbb{R}^n$ and let $0 \leq m \leq n$ be any integer. Then there exists $t \in W_n$ and a $B \subseteq \{1,\ldots,n\}$ of cardinality $m$ such that 
\begin{align*}
N_B(t,J) \geq m!(n-m)! \gamma_n(J).
\end{align*}
\end{lem}

\begin{proof}
Since $\gamma_n(W_n) = 1/n!$, by \eqref{eq:xx} it follows that there exists $t \in W_n$ such that $N(t,J) \geq \gamma_n(J) n!$. Now using \eqref{eq:Bsum} and the fact that there are $\binom{n}{m}$ distinct subsets $B$ of $\{1,\ldots,n\}$ of cardinality $m$, it follows that there exists some $B\subseteq \{1,\ldots,n\}$ of cardinality $m$ such that $N_B(t,J) \geq m!(n-m)! \gamma_n(J) $. 
\end{proof}

% Later we will use the idea that, provided $n$ and $m = pn$ are large, if $\sigma_1,\ldots,\sigma_d$ are i.i.d.\ uniform random permutations in $\mathcal{S}_n(B)$ (for any subset $B$ of $\{1,\ldots,n\}$ of cardinality $m$), then the coupling measure $C^{\bm \sigma}$ associated with this $d$-tuple is likely to be close to copula $\pi_p$ on $[0,1]^d$ given in \eqref{eq:boxcop}. 

\subsection{Mean and variance estimates for random coupling measures}
In the course of our proof of Theorem \ref{thm:lower}, we will need to appeal to concentration properties of random coupling measures associated with random $d$-tuples of permutations. We begin with the following lemma:

 \begin{lem} \label{lem:pc}
 Let $d \leq m$. Let $(\tau_1,\ldots,\tau_d)$ be a $d$-tuple of independent random permutations, such that each $\tau_i$ is uniformly distributed on $\mathcal{S}_m$. Consider the associated random density \begin{align*}
 C^{\bm \tau}(r) := m^{d-1}\sum_{j=1}^m \prod_{i=1}^d \mathrm{1}\left\{ r_i \in  \left[ \frac{\tau_i^{-1}(j)-1}{m},\frac{\tau_i^{-1}(j)}{m} \right) \right\}
 \end{align*}
 on $[0,1]^d$. For measurable $A \subseteq [0,1]^d$ let $C^{\bm \tau}(A) := \int_A C^{ \bm \tau}(r) \mathrm{d}r$ be the associated measure of $A$, and let $|A|$ denote its Lebesgue measure. Then the expectation and variance of $C^{\bm \tau}(A)$ satisfy
 \begin{align} \label{eq:smart}
 \mathbf{E}[ C^{\bm \tau}(A) ] = |A| \quad \text{and} \quad \mathbf{E}[ (C^{ \bm \tau}(A) - |A|)^2 ] \leq C \frac{d}{m} |A|,
 \end{align}
 where $C$ is a universal constant not depending on $m, d$ or $A$.
 \end{lem}
 \begin{proof}
 We begin by proving $\mathbf{E}[ C^{\bm \tau}(A) ] = |A| $. Given $x = (x_1,\ldots,x_d) \in [m]^d := \{1,\ldots,m\}^d$, write 
 \begin{align*}
 A_x := \left\{ r \in A : r_i \in \left[ \frac{x_i-1}{m}, \frac{x_i}{m} \right) ~ \forall 1 \leq i \leq d \right\},
 \end{align*}
 so that $A$ can be written as a disjoint union $A = \sqcup_{x \in [m]^d} A_x$, and thus $\sum_{x \in [m]^d } C^{\bm \tau}(A_x) = C^{\bm \tau}(A)$. Let $a_x = |A_x|$ be its Lebesgue measure. Write $\{ \bm \tau^{-1}(j) = x \} :=  \{ \tau_1^{-1}(j) = x_1,\ldots, \tau_d^{-1}(j) = x_d \}$. Then for each $1 \leq j \leq m$ and each $x \in [m]^d$ we have $\mathbf{P}(  \bm \tau^{-1}(j) = x  ) = m^{-d}$. Using \eqref{eq:cd2} we have  
 \begin{align*}
 \mathbf{E}[ C^{\bm \tau}(A_x)] = m^{d-1} a_x \sum_{j=1}^m \mathbf{P}(  \bm \tau^{-1}(j) = x  ) = a_x.
 \end{align*}
 By linearity, it follows that $\mathbf{E}[ C^{\bm \tau}(A) ] = \sum_{x \in [m]^d} a_x = |A|$. 

 We turn to the calculation of the variance of $C^{\bm \tau}(A)$. Here we have 
 \begin{align} \label{eq:pintor}
 \mathbf{E}[ C^{\bm \tau}(A)^2 ] &= \sum_{ x, y \in [m]^d} \mathbf{E}[ C^{\bm \tau}(A_x) C^{\bm \tau}(A_y) ] \nonumber \\
 &= \sum_{x, y \in [m]^d} \sum_{1 \leq j,k \leq m} m^{2(d-1)} a_xa_y \mathbf{P} \left( \bm \tau^{-1}(j) = x, \bm \tau^{-1}(k) =  y  \right).
 \end{align}
 
 For $x,y \in [m]^d$, write $x \sim y$ if there exists $1 \leq i \leq d$ such that $x_i = y_i$, and write $x \nsim y$ otherwise. (Note $\sim$ is not an equivalence relation.) Then for all $x,y \in [m]^d$ and $1 \leq j,k \leq m$ we have  
 \begin{align} \label{eq:pintor2}
  \mathbf{P} \left( \bm \tau^{-1}(j) = x, \bm \tau^{-1}(k) = y  \right) =
 \begin{cases}
 m^{-d} \quad &\text{if $j = k,x=y$},\\
 0 \quad &\text{if $j=k, x \neq y$},\\
  m^{-d}(m-1)^{-d} \quad &\text{if $j \neq k$, $x \nsim y$},\\
 0 \qquad &\text{if $j \neq k$, $x \sim y$}.
 \end{cases}
 \end{align}
 
 Using \eqref{eq:pintor2} in \eqref{eq:pintor} we have 
 \begin{align} \label{eq:pintor3}
 \mathbf{E}[ C^{\bm \tau}(A)^2 ] &=  \sum_{j=1}^m m^{2(d-1)} \sum_{x \in [m]^d} a_x^2 m^{-d} + \sum_{1 \leq j \neq k  \leq m} m^{2(d-1)} \sum_{ x \nsim y } a_x a_y m^{-d}(m-1)^{-d} \nonumber \\
 &=   m^{d-1} \sum_{x \in [m]^d} a_x^2 + (1-1/m)^{-(d-1)} \sum_{ x \nsim y} a_x a_y ,
 \end{align}
 where $\sum_{ x \nsim y}$ denotes the sum taken over all pairs of elements $x$ and $y$ in $[m]^d$ satisfying the restriction that $x \nsim y$.  Lifting this restriction, and using $\sum_{x \in [m]^d} a_x = |A|$ together with the fact that $a_x \leq m^{-d}$, we obtain from \eqref{eq:pintor3}
  
 the upper bound
 \begin{align} \label{eq:pintor4}
 \mathbf{E}[ C^{\bm \tau}(A)^2 ] \leq \frac{1}{m}|A| + (1 - 1/m)^{-(d-1)} |A|^2.
 \end{align}
 Since $d/m \leq 1$, we have $(1-1/m)^{-(d-1)} \leq 1 + C d/m$. Using this fact in \eqref{eq:pintor4}, subtracting $\mathbf{E}[C^{\bm \tau}(A)]^2 = |A|^2$, and then using the fact that $|A| \leq 1$ to obtain the final inequality below we have
  \begin{align} \label{eq:pintor5}
 \mathbf{E}[ (C^{\bm \tau}(A)-|A|)^2 ] \leq \frac{1}{m}|A| + C \frac{d}{m} |A|^2 \leq C' \frac{d}{m} |A|,
 \end{align}
 completing the proof of \eqref{eq:smart}.

 \end{proof}

Using Chebyshev's inequality, we are able to prove the following bound controlling the deviations of the random variable $C^{\bm \tau}(A)$:
\begin{cor} \label{cor:pc}
Under the conditions of Lemma \ref{lem:pc} we have 
\begin{align*}
\mathbf{P} ( C^{\bm \tau}(A)  \geq 2 \sqrt{|A|} ) \leq C \frac{d}{m}. 
\end{align*}
\end{cor}
\begin{proof}
Using Chebyshev's inequality to obtain the first inequality below, and then \eqref{eq:smart} to obtain the second, we have 
\begin{align} \label{eq:tcalm}
\mathbf{P}\left( |C^{\bm \tau}(A) - |A| | \geq t \right) \leq \frac{1}{t^2} \mathbf{E}[ (C^{\bm \tau}(A) - |A| )^2 ] \leq  C \frac{d}{m} \frac{|A|}{t^2}. 
\end{align}
Now using the fact that $\sqrt{|A|} \geq |A|$ to obtain the first inequality below, then \eqref{eq:tcalm} with $t = \sqrt{|A|}$ to obtain the second, we have 
\begin{align*}
\mathbf{P}\left( C^{\bm \tau}(A) \geq 2 \sqrt{|A|} \right) \leq \mathbf{P}\left( | C^{\bm \tau}(A) - |A| | \geq \sqrt{|A|} \right) \leq C \frac{d}{m},
\end{align*}
completing the proof.
\end{proof}

We close this section with a couple of remarks on relationships between random coupling measures associated with independent uniform elements of $\mathcal{S}_n(B)$ and $\mathcal{S}_m$. Consider the affine mapping $T:[1-p,1]^d \to [0,1]^d$ that sends the $i^{\text{th}}$ coordinate $r_i \in [1-p,1]$ to $\frac{1}{p}(r_i - (1-p))$. Given a subset $\Gamma$ of $[1-p,1]^d$, write $\tilde{\Gamma} := T(\Gamma)$ for the image of $\Gamma$ under this map. 

We make the following observation:
\begin{remark} \label{rem:unif}
Suppose that $\sigma_1,\ldots,\sigma_d$ are independent random permutations that are uniformly distributed on $\mathcal{S}_n(B)$ for some subset $B$ of $\{1,\ldots,n\}$ of cardinality $m=pn$. Then given any subset $\Gamma$ of $[1-p,1]^d$ we have the equality in distribution 
\begin{align} \label{eq:dist}
C^{ \bm \sigma}( \Gamma) \stackrel{(d)}{=} p ~C^{\bm \tau}(\tilde{\Gamma})
,
\end{align}
where $\bm \tau = (\tau_1,\ldots,\tau_d)$ is a $d$-tuple of independent uniform random permutations in $\mathcal{S}_m$. 
\end{remark}
This remark follows from the fact that if $\sigma_1,\ldots,\sigma_d$ are independent uniform random elements of $\mathcal{S}_n(B)$, then when restricted to $\{n-m+1,\ldots,n\}$, the permutations $\sigma_2^{-1} \sigma_1,\ldots,\sigma_d^{-1}\sigma_1$ are independent and uniformly distributed on the set of permutations on $\{n-m+1,\ldots,n\}$.

%%%%%%%%%%%
%%%%%%%%%%%%%%%%%%%%%%%%%%%%%%%%%%%%%%%%%%
\subsection{Exceedances of random probability measures}
%%%%%%%%%%%%%%%%%%%%%%%%%%%%%%%%%%%%%%%%%%%%%%%%%%%%%
\label{sec:randomswitch}
Let $\bm \sigma := (\sigma_1,\ldots,\sigma_d)$ be a (possibly random) $d$-tuple of permutations in $\mathcal{S}_n$. Given a probability measure $\mu$ on $\mathbb{R}$ with quantile function $Q$, we write $\mu^{\bm \sigma}$ for the new probability measure on $\mathbb{R}$ that is the law of the random variable
\begin{align*}
\frac{1}{d}( Q(U_1) + \cdots + Q(U_d))
\end{align*}
where $(U_1,\ldots,U_d)$ is distributed according to $C^{\bm \sigma}$. If $\bm \sigma$ is a random $d$-tuple of permutations, then $\mu^{\bm \sigma}$ is a random probability measure on $\mathbb{R}$. 

Recall that $\sigma s = (s^{\sigma(1)},\ldots,s^{\sigma(n)})$. 
Suppose that $t =(t^1,\ldots,t^n)$ is a vector with empirical coordinate measure $\mu_t$. Then by Lemma \ref{lem:probrep}, $\mu_t^{\bm \sigma}$ is the empirical coordinate measure of the vector $\frac{1}{d} (\sigma_1t + \cdots + \sigma_d t)$. That is,
\begin{align} \label{eq:trotate}
\mu_t^{\bm \sigma} = \mu_{d^{-1}(\sigma_1 t+ \cdots + \sigma_d t )}.
\end{align}

Our next result describes the likely exceedances of the random probability measure $\mu^{\bm \sigma}$ associated with a $d$-tuple of random permutations sampled from some $\mathcal{S}_n(B)$ when $\mu$ is close to the standard Gaussian law $\gamma$.

\begin{prop} \label{prop:random}
Let $B$ be a subset of $\{1,\ldots,n\}$ of cardinality $m = (p_1-\rho)n$, where $0 < \rho < c$. Let $\delta$ be such that $\sqrt{\delta} \leq (\kappa/2) \rho$. (Here $c,\kappa$ are as in the statement of Lemma \ref{lem:boxdiag}.)

Let $\mu$ be a probability measure on $\mathbb{R}$ satisfying $W(\mu,\gamma) \leq \delta$. 

Let $\sigma_1,\ldots,\sigma_d$ be a $d$-tuple of independent random permutations that are uniformly distributed on $\mathcal{S}_n(B)$. Then 
\begin{align*}
\{ \mathrm{Exc}_1(\mu^{\bm \sigma}) \geq p_1 - C ( \rho + e^{ - \frac{c}{2}\rho^2 d})\} \quad \text{with probability at least $1-Cd/n$}.
\end{align*}
\end{prop}

\begin{proof}
By Lemma \ref{lem:stable}, if $\mu$ and $\tilde{\mu}$ are probability measures on $\mathbb{R}$, then 
\begin{align*}
W(\mu^{\bm \sigma},\tilde{\mu}^{\bm \sigma}) \leq W(\mu,\tilde{\mu}).
\end{align*}
In particular, $W(\mu^{\bm \sigma}, \gamma^{\bm \sigma}) \leq \delta$. Now applying Lemma \ref{lem:wasser} we have
\begin{align} \label{eq:w1}
\mathrm{Exc}_1(\mu^{\bm \sigma}) \geq \mathrm{Exc}_{1+\sqrt{\delta}}(\gamma^{\bm \sigma}) - \sqrt{\delta}.
\end{align}
Now 
\begin{align} \label{eq:w2}
\mathrm{Exc}_{1 + \sqrt{\delta}}( \gamma^{\bm \sigma}) &:= \int_{[0,1]^d}  \mathrm{1} \left\{ d^{-1}(\Phi^{-1}(r_1) + \cdots + \Phi^{-1}(r_d)) \geq 1 + \sqrt{\delta} \right\} C^{\bm \sigma}(r) \mathrm{d}r \nonumber \\
& \geq \int_{[1-p_1+\rho,1]^d} \mathrm{1} \left\{ d^{-1}(\Phi^{-1}(r_1) + \cdots + \Phi^{-1}(r_d)) \geq 1 + \sqrt{\delta} \right\} C^{\bm \sigma}(r)  \mathrm{d}r \nonumber \\
& = p_1 - \rho - C^{ \bm \sigma}( \Gamma_\delta),
\end{align}
where
\begin{align*}
\Gamma_\delta := \{ r \in [1-p_1+\rho,1]^d : d^{-1}(\Phi^{-1}(r_1) + \cdots + \Phi^{-1}(r_d)) <  1 + \sqrt{\delta} \},
\end{align*}
and we have used the fact that $C^{\bm \sigma}([1-p_1+\rho,1]^d) = p_1 - \rho$.

By \eqref{eq:w1}, \eqref{eq:w2}, and the fact that $\sqrt{\delta} \leq (\kappa/2)\rho$, we have
\begin{align} \label{eq:w10}
\mathrm{Exc}_1(\mu^{\bm \sigma}) \geq p_1 - C \rho - C^{\bm \sigma}(\Gamma_\delta),
\end{align}
for some sufficiently large universal $C>0$.

Setting $p = p_1 - \rho$ in Remark \ref{rem:unif}, we have 
\begin{align} \label{eq:dist2}
C^{ \bm \sigma}( \Gamma_\delta) \stackrel{(d)}{=} (p_1 -\rho) C^{\bm \tau}(\tilde{\Gamma}_\delta)
,
\end{align}
where $\bm \tau = (\tau_1,\ldots,\tau_d)$ is a $d$-tuple of independent random permutations uniformly distributed on $\mathcal{S}_m$, and $\tilde{\Gamma}_\delta$ is the image of $\Gamma_\delta$ under the affine map $[1-p_1+\rho,1]^d \mapsto [0,1]^d$. 

Now observe that 
\begin{align*}
|\tilde{\Gamma}_\delta| = P_\rho ( d^{-1}(Z_1 + \cdots + Z_d) \leq 1 + \sqrt{\delta} ),
\end{align*}
where, as in the statement of Lemma \ref{lem:boxdiag}, $P_\rho$ governs a $d$-tuple of independent random variables $(Z_1,\ldots,Z_d)$ each of which has the law of a standard Gaussian random variable conditioned to exceed $\Phi^{-1}(1 - p_1+\rho)$.

Using the fact that $\sqrt{\delta} \leq \frac{\kappa}{2} \rho$, by Lemma \ref{lem:boxdiag} we have 
\begin{align} \label{eq:tildy}
|\tilde{\Gamma}_\delta| \leq e^{ - c \rho^2 d}.
\end{align}

Using \eqref{eq:dist2} to obtain the first inequality below, \eqref{eq:tildy} to obtain the second, and then applying Corollary \ref{cor:pc} to $C^{\bm \tau}(\tilde{\Gamma}_\delta)$ to obtain the third, we have
\begin{align} \label{eq:neartime}
\mathbf{P} \left( C^{\bm \sigma}(\Gamma_\delta) \geq 2 e^{ - \frac{c}{2} \rho^2 d } \right) \leq \mathbf{P} \left( C^{\bm \tau}(\tilde{\Gamma}_\delta) \geq 2 e^{ - \frac{c}{2} \rho^2 d } \right) \leq \mathbf{P} \left( C^{\bm \tau}(\tilde{\Gamma}_\delta) \geq 2 \sqrt{|\tilde{\Gamma}_\delta|} \right) \leq C \frac{d}{n},
\end{align}
where we have used the fact that $m > c n$ in the final inequality above.

Combining \eqref{eq:neartime} and \eqref{eq:w10}, and using the fact that $\sqrt{\delta} \leq (\kappa/2) \rho$, we obtain the result with $C = 1 + \frac{\kappa}{2}$. 
\end{proof}

\subsection{Proof of Theorem \ref{thm:lower}}
We are now ready to wrap our work together to prove Theorem \ref{thm:lower}, which states that if $K$ is a convex subset of $\mathbb{R}^n$ with $\gamma_n(K) \geq \varepsilon$, there exists a vector $u$ in $K$ whose empirical coordinate measure satisfies $\mathrm{Exc}_1(\mu_u) \geq p_1 - C_\varepsilon \log(n)^{-1/3}$. 

We outline the proof strategy. Given an element $t$ of $W_n$, and a random $d$-tuple of permutations in some $\mathcal{S}_n(B)$, we define a vector
\begin{align*}
u := \frac{1}{d} ( \sigma_1t + \cdots + \sigma_d t).
\end{align*}
We are going to show in the course of the proof that for a certain choice of $t$ and $B$, the vector $u$ lies in $K$ with not too small a probability. We are also going to show that the exceedance of this random vector is high with high probability. Combining these two observations, we conclude that there must exist a vector $u$ in $K$ whose exceedance is high.

\begin{proof}[Proof of Theorem \ref{thm:lower}]
Let $n \in \mathbb{N}$ and $\varepsilon \in (0,1/2]$. We note from the statement that we may assume without loss of generality that $n \geq C/\varepsilon$ for a sufficiently large universal constant $C > 0$. We will consider a convex subset $K$ of $\mathbb{R}^n$ whose Gaussian measure satisfies $\gamma_n(K) \geq \varepsilon$.

To set up our proof, we define an integer 
\begin{align} \label{eq:ddef}
d = d_{n,\varepsilon} = \lfloor c \log(n) / \log(1/\varepsilon) \rfloor,
\end{align}
where $c$ is a sufficiently small constant to be determined below. We now define 
 \begin{align} \label{eq:rhodef}
 \rho = \rho_{n,\varepsilon} := \inf \{ r \geq C \sqrt{\log(d_{n,\varepsilon})/d_{n,\varepsilon}} : \text{$n(p_1 - r)$ is an integer} \},
 \end{align}
 where $C$ is a sufficiently large constant also to be determined below. Set 
 \begin{align*}
  m := m_{n,\varepsilon} = n(p_1 - \rho_{n,\varepsilon})
  \end{align*}
 for the remainder of the proof.

Using the definition \eqref{eq:rhodef} to obtain the first inequality below, and then \eqref{eq:ddef} for the second, we have 
\begin{align*}
\rho_{n,\varepsilon} \leq C d_{n,\varepsilon}^{-1/3} \leq C_\varepsilon \log(n)^{-1/3}
\end{align*}
where $C_\varepsilon = C \log(1/\varepsilon)^{1/3}$ for some universal $C> 0$. 

To lighten notation we will write $\rho = \rho_{n,\varepsilon}, d = d_{n,\varepsilon}$, and $m = m_{n,\varepsilon}$ for the remainder of the proof.

Recall $\delta_n := C n^{-1/3}$. 
We saw in Proposition \ref{prop:DKW} that the set $A_n = E_n(\delta_n)$ of vectors $s$ in $\mathbb{R}^n$ whose empirical coordinate measure satisfies $W(\mu_s,\gamma)\leq \delta_n$ satisfies $\gamma_n(A_n) \geq 1 - C/n$. Since $n \geq C/\varepsilon$, by Proposition \ref{prop:DKW} we have $\gamma_n(K \cap A_n) \geq \varepsilon/2$. By Lemma \ref{lem:existence} there exists some $t \in W_n$ and some subset $B$ of cardinality $m = (p_1-\rho)n$ such that $N_B(t, K \cap A_n) \geq (\varepsilon/2)m!(n-m)!$. 

Note that $N_B(t, K \cap A_n) > 0$ implies $\sigma t \in A_n$ for some $\sigma$, which of course implies $t$ itself lies in $A_n$ since $\mu_{\sigma t} = \mu_t$. In particular, $W(\mu_t,\gamma) \leq \delta_n$. 

With this choice of $B$, for the remainder of the proof let $(\sigma_1,\ldots,\sigma_d)$ be a $d$-tuple of independent random permutations each of which is uniformly distributed on $\mathcal{S}_n(B)$. Then each $\sigma_it$ is a random vector in $\mathbb{R}^n$. 

Now on the one hand since $N_B(t, K \cap A_n) \geq (\varepsilon/2)m!(n-m)!$ and $\# \mathcal{S}_n(B) = m!(n-m)!$ we have
\begin{align} \label{eq:R1}
\{ \sigma_i t \text{ lies in } K \cap A_n \text{ for each } 1 \leq i \leq d\}  \text{ has probability at least } (\varepsilon/2)^d.
\end{align}
On the other hand, we would like to apply Proposition \ref{prop:random}. Bearing in mind the statement of Proposition \ref{prop:random}, we note that with $\rho = \rho_{n,\varepsilon}$ and $d = d_{n,\varepsilon}$, we have
\begin{align} \label{eq:rhorel}
\rho + e^{ - c d^2 \rho}  \leq C' \rho \leq C_\varepsilon \log(n)^{-1/3},
\end{align}
provided the constant $C > 0$ in \eqref{eq:rhodef} is sufficiently large. Also note that if $C$ in \eqref{eq:rhodef} is sufficiently large, then we will have $\sqrt{\delta_n} \leq (\kappa/2)\rho_{n,\varepsilon}$. Moreover, $0 < \rho_{n,\varepsilon} < c$ whenever $n \geq C/\varepsilon$. It follows that the conditions of Proposition \ref{prop:random} are satisfied, and hence 
the random probability measure $\mu_t^{\bm \sigma}$ has
\begin{align} \label{eq:R2} 
\left\{ \mathrm{Exc}_1(\mu_t^{\bm \sigma}) \geq p_1 - C_\varepsilon \log(n)^{-1/3} \right\}  \text{ with probability at least $1 - Cd/n$},
\end{align}
where to obtain \eqref{eq:R2} from the statement of Proposition \ref{prop:random}, we have used \eqref{eq:rhorel}.

Now with $d =d_{n,\varepsilon}$ defined in \eqref{eq:ddef} for a sufficiently small value of $c$, we have $Cd/n < (\varepsilon/2)^d$. It follows that
\begin{align*}
1- Cd/n + (\varepsilon/2)^d > 1,
\end{align*}
and accordingly there must be some $d$-tuple $\bm \sigma = (\sigma_1,\ldots,\sigma_d)$ of elements of $\mathcal{S}_n(B)$ such that the events in \eqref{eq:R1} and \eqref{eq:R2} both happen simultaneously. In other words, there exists a $d$-tuple $\bm \sigma = (\sigma_1,\ldots,\sigma_d)$ of permutations such that we have both 
\begin{align} \label{eq:membership}
\sigma_it \in K \cap A_n \text{ for every $i = 1,\ldots,d$}
\end{align}
and
\begin{align} \label{eq:membership2}
\mathrm{Exc}_1(\mu_t^{\bm \sigma}) \geq p_1 - C_\varepsilon \log(n)^{-1/3}.
\end{align}

It is at precisely this stage that we use the convexity of $K$. Now, by \eqref{eq:membership} since each $\sigma_it$ lies in $K$ and $K$ is convex, the vector
\begin{align*}
u := d^{-1}(\sigma_1 t+ \cdots + \sigma_dt) \text{ lies in $K$}.
\end{align*} 
On the other hand, recall that by \eqref{eq:trotate} we have $\mu_t^{\bm \sigma} = \mu_{ d^{-1}(\sigma_1 t + \cdots + \sigma_d t)} = \mu_u$. It thus follows from \eqref{eq:membership2} that there exists $u \in K$ such that
\begin{align*}
\mathrm{Exc}_1(\mu_u) \geq p_1 - C_\varepsilon \log(n)^{-1/3}.
\end{align*}
That completes the proof of Theorem \ref{thm:lower}.

\end{proof}

\appendix
\section{Proof of Lemma \ref{lem:boxdiag}} 
In this appendix we will present the rigorous proof of Lemma \ref{lem:boxdiag}. For this purpose we will require formulas for all first- and second-order mixed partial derivatives evaluated at $(0,0)$ of the function
\begin{align} \label{eq:Sdef}
S(\alpha,\rho) := \int_{ \Phi^{-1}(1-p_1+\rho) }^\infty e^{ - \alpha u} \gamma(u)\mathrm{d}u.
\end{align}
Let us introduce the shorthand $S_\alpha := \frac{\partial}{\partial \alpha} S(\alpha,\rho)|_{(\alpha,\rho)=(0,0)}$ and $S_{\alpha \rho}:= \frac{\partial^2}{\partial \alpha \partial \rho } S(\alpha,\rho)|_{(\alpha,\rho)=(0,0)}$, and similarly for partial derivatives evaluated at $(0,0)$. Let $S := S(0,0)$. 

Recall from \eqref{eq:py} that $y_w$ is the unique real number such that $\Phi'(y_w)/\Phi(y_w) = w$, and $p_w = \Phi(y_w)$. 
The following lemma gives formulas for the partial derivatives of $S(\alpha,\rho)$ at zero:

\begin{lem} \label{lem:derivative}
We have
\begin{align} \label{eq:Sform}
S = p_1, \quad S_\rho = -1, \quad S_\alpha = - p_1, \quad S_{\rho \rho} =  0 \quad \text{and} \quad S_{\alpha \rho} = - y_1.
\end{align} 
Moreover, we have
\begin{align} \label{eq:Sform2}
   V :=  S_{\alpha \alpha}/p_1-1 > 0.
    \end{align}
\end{lem}
\begin{proof}

Plainly
\begin{align*}
S = \int_{\Phi^{-1}(1-p_1)}^\infty \gamma(u)\mathrm{d}u = p_1.
\end{align*}
Next, using the definition of $y_1$ to obtain the second equality below, \eqref{eq:rel1} to obtain the third, and then the fact that $y_w$ solves $\Phi'(y_w)/\Phi(y_w) = w$ to obtain the fourth, we have
\begin{align*}
S_\alpha = - \int_{\Phi^{-1}(1-p_1)}^\infty u\gamma(u)\mathrm{d}u = - \int_{-y_1}^\infty u \gamma(u) \mathrm{d}u = - \Phi'(y_1) = - \Phi(y_1) = -p_1.
\end{align*}
As for $S_\rho$, first we search for a more convenient representation for differentiating $S(\alpha,\rho)$ with respect to $\rho$. By setting $f(x) = \mathrm{1}_{\{ x \geq \Phi^{-1}(1 - p_1 + \rho)\}} e^{ - \alpha x}$ in \eqref{eq:qrel}, we may write
\begin{align} \label{eq:qrep}
S(\alpha,\rho) := \int_{1-p_1+\rho}^1 e^{ - \alpha \Phi^{-1}(r) } \mathrm{d}r.
\end{align}
Differentiating \eqref{eq:qrep} with respect to $\rho$ and setting $(\alpha,\rho) = (0,0)$ we obtain
\begin{align*}
S_\rho = -1 \quad \text{and} \quad S_{\rho \rho} = 0.
\end{align*}
Finally, differentiating \eqref{eq:qrep} first with respect to $\rho$, then with respect to $\alpha$, and then subsequently setting $(\alpha,\rho)=(0,0)$, we obtain
\begin{align*}
S_{\alpha \rho} = \Phi^{-1}(1-p_1)  = -y_1,
\end{align*}
completing the last of the derivations of the formulas in \eqref{eq:Sform}.

As for the inequality in \eqref{eq:Sform2}, since $\Phi'(y_1)/\Phi(y_1) = \mathbf{E}[ Z | Z > - y_1 ] = 1$, the quantity 
\begin{align*}
V := S_{\alpha \alpha}/{p_1}-1 = \frac{1}{p_1}\int_{-y_1}^\infty u^2 \gamma(u)\mathrm{d}u - \left( \frac{1}{p_1} \int_{-y_1}^\infty u \gamma(u)\mathrm{d}u \right)^2
\end{align*}
is simply the variance of a standard Gaussian random variable conditioned to exceed $- y_1$, and as such, is positive. 
\end{proof}

We are now ready to proceed with the proof of Lemma \ref{lem:boxdiag}. 
\begin{proof}[Proof of Lemma \ref{lem:boxdiag}]
For any $\alpha \geq 0$ and $w \in \mathbb{R}$ we have the inequality 
\begin{align} \label{eq:cas}
\mathrm{1}\{ d^{-1}(Z_1+\cdots + Z_d) \leq w \} \leq \exp \left\{ \alpha \left( dw - (Z_1+ \cdots + Z_d) \right) \right\}.
\end{align}
Suppose now $0 < \rho < p_1/2$. Then setting $\alpha = \rho x$ and letting $w = 1 + (\kappa/2) \rho$ in \eqref{eq:cas}, and then subsequently taking expectations with respect to $P_\rho$ through this inequality, we obtain the Chernoff inequality
\begin{align} \label{eq:colo0}
P_\rho \left( d^{-1}(Z_1+\cdots+Z_d) \leq 1 + (\kappa/2) \rho \right) \leq \exp \{ d I_{x} (\rho)\},
\end{align}
where with $S(\alpha,\rho)$ as in \eqref{eq:Sdef} we have 
\begin{align} \label{eq:Icx}
I_{x}(\rho) := \rho x ( 1 +(\kappa/2) \rho) + \log \frac{S(\rho x,\rho)}{p_1 - \rho},
\end{align}
and where we note that the ratio 
\begin{align*}
 \frac{S(\alpha,\rho)}{p_1 - \rho} = E_\rho[e^{ - \alpha Z_1} ]
\end{align*}
is the Laplace transform of a standard Gaussian random variable conditioned on exceeding $\Phi^{-1}(1 - p_1 + \rho)$.

We expand $I_{x}(\rho)$ as a power series in $\rho$. 
For each $x$ there exist $c_x, C_x > 0$ such that whenever $0 < \rho < c_x$ we have 
\begin{align} \label{eq:crea}
S(\rho x, \rho) \leq S + \rho ( S_\rho + x S_\alpha ) + \frac{\rho^2}{2} \left( x^2 S_{\alpha \alpha} + 2 x S_{\alpha \rho} +  S_{\rho \rho} \right) + C_x \rho^3,
\end{align}
where we are using the notation of Lemma \ref{lem:derivative}. Using \eqref{eq:Sform} in \eqref{eq:crea} we have
\begin{align} \label{eq:crea2}
S(\rho x, \rho) \leq p_1 + \rho ( -1 - p_1 x ) + \frac{\rho^2}{2} \left( x^2 S_{\alpha \alpha} - 2xy_1 \right) + C_x \rho^3,
\end{align}
where we leave the $S_{\alpha \alpha}$ term as it stands for now. A brief calculation using \eqref{eq:crea2} and \eqref{eq:Icx} tells us that whenever $0 < \rho < c_x$ we have 
\begin{align*}
I_{x}(\rho) \leq \rho^2 \left(- \frac{\kappa}{2} x + \frac{1}{2}Vx^2 \right) + C'_x \rho^3,
\end{align*} 
where as noted in \eqref{eq:Sform2}, $V := S_{\alpha \alpha}/p_1 - 1 > 0$. 
Letting $x_0 = \kappa/2V$ be the value of $x$ for which $- \frac{\kappa}{2} x + \frac{1}{2}Vx^2$ is minimised, for all $0 < \rho < c_{x_0}$ we have 
\begin{align*}
I_{\kappa}(\rho) \leq - \rho^2 \frac{\kappa^2}{8V} + C'_{x_0} \rho^3. 
\end{align*}
In particular, it follows that provided $0 < \rho < c$ for some universal constant $c > 0$ we have
\begin{align} \label{eq:colo} 
I_{x_0}(\rho) \leq - c\rho^2.
\end{align}
Plugging \eqref{eq:colo} into \eqref{eq:colo0} completes the proof of Lemma \ref{lem:boxdiag}.
\end{proof}

\end{document}